\documentclass[12pt, leqno]{amsart}

\usepackage{indentfirst}
\usepackage{amstext}
\usepackage{amsopn}
\usepackage{amsfonts}
\usepackage{amsmath}
\usepackage{latexsym}
\usepackage{amscd}
\usepackage{amssymb}
\usepackage{amsmath}
\usepackage{mathtools}
\usepackage[all,cmtip]{xy}
\usepackage{leftidx}
\usepackage{graphicx}
\usepackage{tikz}
\usepackage{tikz-cd}

=\oddsidemargin \font\teneufm=eufm10 \font\seveneufm=eufm7
\font\fiveeufm=eufm5
\newfam\eufmfam
\textfont\eufmfam=\teneufm \scriptfont\eufmfam=\seveneufm
\scriptscriptfont\eufmfam=\fiveeufm
\def\frak#1{{\fam\eufmfam\relax#1}}
\let\goth\mathfrak

\def\cC{\mathcal C}

\def\cF{\mathcal F}

\def\cR{\mathcal R}
\def\cO{\mathcal O}
\def\cT{\mathcal T}

\def\cE{\mathcal E}
\def\cM{\mathcal M}
\def\cN{\mathcal N}

\def\m{\goth m}
\def\gm{\goth m}

\def\GG{\mathbb{G}}
\def\NN{\mathbb{N}}

\def\WW{\mathbf{W}}
\def\VV{\mathbf{V}}
\def\HH{\frak H}
\def\gg{\goth g}
\def\gh{\goth h}

\def\gD{\goth D}

\def\gg{\goth g}

\def\1{\mbox{\bf 1}}

 \DeclareMathOperator{\Hom}{Hom}
\DeclareMathOperator{\Aut}{Aut}
\DeclareMathOperator{\Autext}{Autext}

\DeclareMathOperator{\Isom}{Isom}
\DeclareMathOperator{\Isomext}{Isomext}
\DeclareMathOperator{\Isomint}{Isomint}
\DeclareMathOperator{\Stab}{Stab}

\DeclareMathOperator{\Spec}{\rm Spec}

\DeclareMathOperator{\GL}{\rm GL}
\DeclareMathOperator{\SL}{\rm SL}

\newcommand{\incl}[1][r]
{\ar@<-0.2pc>@{^(-}[#1] \ar@<+0.2pc>@{-}[#1]}

\newcommand{\imm}[1][r]
   {\ar@{}[#1] |*[o][F]{\hbox{%
         }}
     \ar@{^{(}->}[#1]}

\newcommand{\uA}{{\underline{A}}}

\newcommand{\uC}{{\underline{C}}}
\newcommand{\uE}{{\underline{E}}}
\newcommand{\uF}{{\underline{F}}}
\newcommand{\uG}{{\underline{G}}}
\newcommand{\uH}{{\underline{H}}}
\newcommand{\uM}{{\underline{M}}}
\newcommand{\uN}{{\underline{N}}}
\newcommand{\uQ}{{\underline{Q}}}
\newcommand{\uR}{{\underline{R}}}
\newcommand{\uX}{{\underline{X}}}
\newcommand{\uY}{{\underline{Y}}}
\newcommand{\uW}{{\underline{W}}}

\newtheorem{theorem}{Theorem}[subsection]

\newtheorem{claim}[theorem]{Claim}

\newtheorem{proposition}[theorem]{Proposition}

\newtheorem{stheorem}{Theorem}[section]

\newtheorem{scorollary}[stheorem]{Corollary}
\newtheorem{slemma}[stheorem]{Lemma}
\newtheorem{sproposition}[stheorem]{Proposition}
\newtheorem{sremark}[stheorem]{Remark}
\newtheorem{sremarks}[stheorem]{Remarks}
\newtheorem{sexample}[stheorem]{Example}

\newtheorem{sdefinition}[stheorem]{Definition}

\theoremstyle{definition}
\newtheorem{remark}[theorem]{Remark}

\numberwithin{equation}{section}

\def\NN{\mathbb{N}}
\def\ZZ{\mathbb{Z}}

\def\gE{\mathfrak{E}}

\def\gR{\mathfrak{R}}

\def\gS{\mathfrak{S}}

\def\G{\mathbb G}

\def\bG{\text{\rm \bf G}}

\def\bZ{\rm \bf{Z}}

\def\bD{\text{\rm \bf D}}

\def\cO{\mathcal{O}}

\def\bR{\text{\rm \bf R}}

\def\bt{\mathbf t}

\def\ol{\overline}
\def\id{\text{\rm id}}
\def\fppf{\text{\rm fppf}}

\def\Lie{\mathop{\rm Lie}\nolimits}

\def\2int{\mathop{2\int}\nolimits}
\def\Dyn{\mathrm{\bf Dyn}}
\def\rDyn{\small\mathrm{Dyn}}
\def\uDyn{\underline{\mathrm{Dyn}}}
\def\uDynn{\underline{\mathrm{Dyn}_{\triangledown}}}  %

\def\dim{\mathop{\rm dim}\nolimits}

\def\Spec{\mathop{\rm Spec}\nolimits}

\def\Lie{\mathop{\rm Lie}\nolimits}

\def\Hom{\mathop{\rm Hom}\nolimits}

\def\Stab{\mathop{\rm Stab}\nolimits}

\def\Gal{\mathop{\rm Gal}\nolimits}

\def\Mat{\mathop{\rm M}\nolimits}

\def\Aut{\mathrm{Aut}}
\def\uAut{\underline{\Aut}}
\def\Autext{\mathrm{Autext}}
\def\uAutext{\underline{\Autext}}
\def\Isom{\mathrm{Isom}}
\def\uIsom{\underline{\Isom}}
\def\Isomext{\mathrm{Isomext}}
\def\uIsomext{\underline{\Isomext}}

\def\uStab{\underline{\Stab}}

\def\Isomint{\mathrm{Isomint}}
\def\uIsomint{\underline{\Isomint}}

\def\Transp{\mathrm{Transp}}
\def\uTransp{\underline{\Transp}}

\def\Centr{\mathrm{Centr}}
\def\uCentr{\underline{\Centr}}

\def\Dyn{\mathrm{Dyn}}
\def\uDyn{\underline{\Dyn}}

\def\Isom{\mathop{\rm Isom}\nolimits}
\def\Isomext{\mathop{\rm Isomext}\nolimits}

\def\resp.{\mathop{\rm resp.}\nolimits}

\def\lgr{\longrightarrow}

\font\math=cmmi10
\def\varpi{\hbox{\math\char'44}}

\def\simlgr{\buildrel\sim\over\lgr}

\def\pa{\S\kern.15em }

\def\un{\uppercase\expandafter{\romannumeral 1}}
\def\deux{\uppercase\expandafter{\romannumeral 2}}
\def\trois{\uppercase\expandafter{\romannumeral 3}}
\def\quatre{\uppercase\expandafter{\romannumeral 4}}
\def\cinq{\uppercase\expandafter{\romannumeral 5}}
\def\six{\uppercase\expandafter{\romannumeral 6}}

\def\gg{\goth g}

\title[Homogeneous spaces]{Oriented embedding functors of tori as homogeneous spaces }

\date{\today}

\author{Philippe Gille}

\address[]{P. Gille, Institut Camille Jordan - Universit\'e Claude Bernard Lyon 1
43 boulevard du 11 novembre 1918,
69622 Villeurbanne cedex - France }
\email{gille@math.univ-lyon1.fr}

\author{Ting-Yu Lee}
\address[]{T.-Y. Lee, Astronomy Mathematics Building 5F, No. 1, Sec. 4, Roosevelt Rd., Taipei 10617, Taiwan.
}
\email{tingyulee@ntu.edu.tw}

\begin{document}

 \begin{abstract} We provide a characterization of homogeneous spaces under a reductive group 
 scheme such that the geometric stabilizers are maximal tori. 
 The quasi-split case over a semilocal base (and more generally over a LG-ring)
 is of special 
 interest and permits to   answer a  question raised by Marc Levine 
 on $\SL_n$-homogeneous spaces. At the end, we provide an application to the local-global principles for embeddings of \'etale algebras with involution into central simple algebras with involution.
 
\end{abstract}

\maketitle

\bigskip


\smallskip


\bigskip

\bigskip

\section{Introduction}
Let $k$ be a field and  we fix an separable closure  $k_s$ of $k$ throughout this article;
we denote by $\Gamma_k=\Gal(k_s/k)$ the absolute Galois group of $k$.
Let $X$ be an affine $k$--variety 
that is a $\SL_n$-homogeneous space such that the stabilizer
of a geometric point is a maximal torus. Marc Levine asked whether 
$X$ is $G$-isomorphic to $\SL_n/T$ where $T$ is a maximal torus of $\SL_n$.

Our main classification result (Theorem \ref{thm_main})
tells us that $X$ is the variety\footnote{Those varieties (and their generalizations over a base) have been defined by the second author in \cite{L1}.} of embeddings of a suitable oriented torus $T$ in $\SL_n$ and 
we explain now the meaning of $X(k)$.
Denoting by $\GG_m^{n,1}$ the standard split torus
of $\SL_n$, we remind  the reader that the automorphism
group of the root system $A_{n-1}=\Phi(\SL_{n}, \GG_{m}^{n,1})$
is $S_n \times \ZZ/2\ZZ$ if $n \geq 3$ and $\ZZ/2\ZZ$ for $n=2$.
We have that $T=R^1_{A/k}(\GG_m)$
for an \'etale $k$--algebra $A$ 
(see \cite[\S 7.5]{G}) and that $X(k)$ corresponds to the  set of embeddings $f: T \to \SL_n$ such that the class of the 
$S_n \times \ZZ/2\ZZ$--torsor $\uIsom\Bigl(\Phi(\SL_{n}, \GG_{m}^{n,1}),\Phi(\SL_{n}, f(T))\Bigr)$ is $[A] \times 1$
if $n\geq 3$ and simply $[A]$ if $n=2$.
We used implicitly there that the Galois cohomology set $H^1(k,S_n)$ classifies isomorphism classes of 
\'etale algebras of degree $n$
and that the oriented type of $T$ is 
the class of the above torsor.

On the other hand, if we see $A$ as a $k$-vector space of dimension $n$,
we have an isomorphism $\SL(A) \simlgr \SL_n$
and the natural embedding $T = R^1_{A/k}(\GG_m) \hookrightarrow \SL_1(A)$
defines then an embedding $T \to \SL_n$ of oriented type $[A] \times 1$.
In other words, the $\SL_n$-homogeneous space $X$ carries a $k$--point whose stabilizer is $T = R^1_{A/k}(\GG_m)$. Thus $X$ is $\SL_n$-isomorphic
to $\SL_n/T$ as desired.

This statement can be strengthened in two 
ways. 
First we can replace $\SL_n$ by any quasi-split reductive
$k$-group, the presence of $k$-points on such a 
homogeneous space being the Gille-Raghunathan's theorem
\cite{Gi2004,Rg}.

With more effort, we can replace further the base field $k$ 
by an arbitrary  LG ring
(with  infinite residue fields).
 This is Corollary   \ref{cor_main} which involves
 a generalization of  the existence of maximal tori of any orientation in such group schemes (see Theorem  \ref{thm_semi_local}). 
 As in Steinberg's section theorem, the case of type $A_{2n}$
 requires an additional argument.
Also this generalization needs the notion of versal torsor
in that setting (Reichstein-Tossici \cite{RT});
we introduced a variation of this technique involving 
algebraic spaces (see \S \ref{subsec_versal}).
 
Finally we provide an arithmetic application to the embeddings of \'etale algebras with involution into central simple algebras with involution.
Let $(E,\sigma)$ be an \'etale algebra $E$  over $k$ with involution $\sigma$ and $(A,\tau)$ be a central simple algebra $A$ over $k$ with involution $\tau$.
We want to know when $(E,\sigma)$ can be embedded into $(A,\tau)$ under some constraints on the dimensions of $(E,\sigma)$ and $(A,\tau)$ (see \cite{PR1}, \cite{BLP1}, \cite{BLP2}).
In Theorem \ref{embed_algebras}, we show that when $k$ is a number field and the unitary group associated to $(A,\tau)$ is quasi-split, the local-global principle always holds for the embeddings of 
$(E,\sigma)$ into $(A,\tau)$.

 \medskip

\noindent{\bf Acknowledgments.} We thank Marc Levine for useful discussions and 
Skip Garibaldi for his suggestion to deal with LG rings. We thank the Camille Jordan  Institute for inviting  the second author.

\vskip1cm

\noindent{\bf Notation.}
We use mainly  the terminology and notations of Grothendieck-Dieudonn\'e \cite[\S 9.4  and 9.6]{EGA1}
 which agree with that  of Demazure-Grothendieck used in \cite[Exp.\ I.4]{SGA3}

(a) Let $S$ be a scheme and let $\cE$ be a quasi-coherent sheaf over $S$.
 For each morphism  $f:T \to S$,
we denote by $\cE_{T}=f^*(\cE)$ the inverse image of $\cE$
by the morphism $f$.
 We denote by $\VV(\cE)$ the affine $S$--scheme defined by
$\VV(\cE)=\Spec\bigl( \mathrm{Sym}^\bullet(\cE)\bigr)$;
it is affine  over $S$ and
represents the $S$--functor $Y \mapsto \Hom_{\cO_Y}(\cE_{Y}, \cO_Y)$
\cite[9.4.9]{EGA1}.

\smallskip

(b) We assume now that $\cE$ is locally free of finite rank and denote by $\cE^\vee$ its dual.
In this case the affine $S$--scheme $\VV(\cE)$ is  of finite presentation
(ibid, 9.4.11); also
the $S$--functor $Y \mapsto H^0(Y, \cE_{Y})=
\Hom_{\cO_Y}(\cO_Y, \cE_{Y} )$
is representable by the  affine $S$--scheme $\VV(\cE^\vee)$
which is also denoted by  $\WW(\cE)$  \cite[I.4.6]{SGA3}.

It applies to the locally free coherent sheaf
${\cE}nd(\cE) = \cE^\vee \otimes_{\cO_S} \cE$
 over $S$ so that we can consider
the affine $S$--scheme $\VV\bigl({\cE}nd(\cE)\bigr)$
that is an $S$--functor in associative commutative and unital algebras
\cite[9.6.2]{EGA1}.
Now we consider the $S$--functor $Y \mapsto \Aut_{\cO_Y}(\cE_{Y})$.
It is representable by an open $S$--subscheme of $\VV\bigl({\cE}nd(\cE)\bigr)$
which is denoted by $\GL(\cE)$ ({\it loc. cit.}, 9.6.4).

\section{Generalities on homogeneous spaces}

Let $S$ be a scheme and let $\bullet$ be the final object of the
category of fppf $S$-sheaves of groups, that is
$\bullet(T)= \Hom_{S-sch}(T,S)=h_S(T)$ for each $S$--scheme $T$.

 Let $\uG$ be a fppf $S$-sheaf of groups.
Let $\uX$ be a fppf $S$-sheaf equipped with  a left action of  $\uG$.
We say that $\uX$ is homogeneous  (resp.\, principal homogeneous) under $\uG$ if
the map $\uX \to \bullet$ is an epimorphism of fppf $S$-sheaves
and if the action map $\uG \times_S \uX \to \uX \times_S \uX$,
$(g,x) \to (x,g.x)$ is an  epimorphism (resp.\ an isomorphism) of fppf $S$-sheaves.
Similarly we have the notion of right homogeneous (resp.\, principal homogeneous) spaces.
We denote by $H^1_{right}(S,\uG)$ the isomorphism classes of principal homogeneous spaces under right $G$-action (\cite[III 2.4.2]{Gir}). 

Let $\uE$ be a right principal homogeneous  $\uG$-space and denote by $\uG'$ the twist of
$\uG$ by $\uE$ through inner automorphisms.
Then twisting by $\uE$ gives rise to an equivalence of categories
between the category of left homogeneous $\uG$--spaces and
that of  left homogeneous $\uG'$--spaces.

\begin{sremark}{\rm For the discussion of the same material in  the representable case, see \cite[VI.1]{R}.
 }
\end{sremark}

\begin{slemma}\label{lem_quotient} Let $\uH$ be a fppf group subsheaf of $\uG$ and $\uX$ be the $\uG$-homogeneous space $\uG/\uH$.  Consider
the $S$--sheaf of groups $\uA=  \uAut_\uG(\uX)$ that acts on the left of $\uX$.

\smallskip

\noindent (1)  Let $\uN=\uN_\uG(\uH)$ be the normalizer of $\uH$ in $\uG$ (as in
\cite[\S 2.3]{Gi2015}).
Then the map $\uG \times_S \uN \to \uG$, $(g,n) \mapsto g n$
 induces an isomorphism  $\uN/\uH \simlgr \uA^{op}$.

\smallskip

\noindent (2) The action of $\uA$  on $\uX$ is free and the map $\uX \to \uA\backslash\uX$
is a (left) principal $\uA$--homogeneous space.

\smallskip

\noindent (3) Let $\uF$ be a right $\uA$-torsor and let $\uX'$ be the twist of $\uX$ by $\uF$.
Then $X'$ is a $G$-homogeneous space and  there is a canonical isomorphism $\uA'\backslash\uX'  \simlgr  \uA\backslash\uX$, where $\uA'=\uAut_\uG(X')$.

\smallskip

\noindent (4) The $S$--forms of $\uX$ (as $\uG$-spaces) are classified by $H^1_{right}(S,  \uA)$.

\end{slemma}

\begin{proof}
We put $\uW=\uN/\uH$. 
Let $1_G$ be the identity element of $\uG(S)$. Denote by $[1]$ the image of $1_G$ in $\uX(S)$.

\noindent{(1)} The map  $\uG \times_S \uN \to \uG$
induces a map $\uX \times \uW \to \uX$ hence a group homomorphism
$\phi: \uW \to \uA^{op}$ of $S$-sheaves of groups.

\smallskip

\noindent{\it $\phi$ is a monomorphism.} Let $S'$ be an $S$-scheme and let
$w \in \uW(S')$ be an element such that $\phi(w)=1$. Up to localization we can assume
that $w$ arises from an element $n \in \uN(S')$ which satisfies $[1] \cdot n =[1]\in \uX(S')$.
Hence $n\in\uH(S')$ and $w=1$.

\smallskip 

\noindent{\it $\phi$ is an epimorphism.} Let $S'$ be an $S$--scheme and let
$f \in \uA(S')$.  
 We put  $x=f([1]) \in \uX(S')$ and up to localization we can assume
that $x$ arises from some $g \in \uG(S')$. We claim that $g \in \uN(S')$.
For each $S'$-scheme $S''$ and   each $g' \in \uG(S'')$  we have
$f( g'\cdot [1])= g' \cdot g_{S''}\cdot [1]$.
In particular for each $h' \in \uH(S'')$, we have $h' \, g_{S''} \in g_{S''} \uH(S'')$
so that $g_{S''}$ normalizes $\uH(S'')$. The claim is proven
and it follows that $f= \phi(w)$ where $w$ is the image of $g \in \uN(S')$ in $\uW(S')$.

\smallskip

\noindent (2) We start by establishing that 
$\uA$ acts freely on $\uX$.
By (1), it is equivalent to proving that $\uW$ acts freely on $\uX$.
Let $S'$ be an $S$-scheme and $x\in \uX(S')$.
After localization we can assume that $x$ arises from an element $g\in\uG(S')$.
Consider the $S'$-sheaf of groups $\uStab_{\uW}(x)$.
For each $S'$-scheme $S''$ and $w\in \uStab_{\uW}(x)(S'')$, up to localization we can assume that $w$ comes from an element $n\in \uN(S'')$.
Then $[g]= [g]\cdot w=[gn]$, that implies $n\in \uH(S'')$ and $w$ is the identity in $\uW(S'')$.
Hence $\uW$ acts freely on $\uX$.
We put $\uY=\uA\backslash\uX$.
In view of \cite[III.3.1.2]{Gir},
the map $\uX\to \uY$ satisfies then  the second part
of the definition of $\uA$-torsor, that is,
the map  $\uA_{\uY}  \times_{\uY} \uX \to  
\uX \times_{\uY} \uX$, $(a,x) \mapsto (x,ax)$,  is an isomorphism.
Since $\uX \to \uY$ admits locally sections for the fppf topology,
the first requirement is satisfied as well \cite[definition III.4.1]{Gir}.
Thus $\uX\to \uY$ is a left $\uA$-torsor.

\smallskip

\noindent (3)
Let $p$ be the projection from $\uX$ to the quotient sheaf $\uA\backslash\uX$.
Consider the map $\pi:\uF\times \uX\to \uA\backslash \uX$, which projects $\uF\times \uX$ to $\uX$ and then to the quotient sheaf $\uA\backslash \uX$.
Then clearly $\pi$ induces a map from $\uX'$ to $\uA\backslash \uX$, which we still denote by $\pi$.

Note that $\uG$ acts on the $\uX$-factor of $\uF\times\uX$. As the $G$-action and the $A$-action commute on $X$,  this defines an $\uG$-action on $\uX'$.
Choose a $fppf$-cover $\{S_i\}$of $S$ that trivializes the $\uA$-torsor $\uF$.
For each $S_i$, there is an isomorphism of right $\uA$-torsors between $\uF_{S_i}$ and $\uA_{S_i}$.
Hence $X'_{S_i}\simeq X_{S_i}$ and $\uA'_{S_i}\simeq \uA_{S_i}$.

Fix an isomorphism of $\uA_{S_i}$-torsors $\iota_i: \uF_{S_i}\to \uA_{S_i}$.
The map $\iota_i$ induces an isomorphism between $\uX'_{S_i}$ and $\uX_{S_i}$ that sends $[f, x]$ to $\iota_i(f)(x)$ for $x\in \uX(T_i)$, $f\in\uF(T_i)$ and for arbitrary $S_i$-scheme $T_i.$
We denote this induced isomorphism still by $\iota_i$.
The map $\iota_i$ induces an isomorphism between $\uA'_{S_i}$ and $\uA_{S_i}$ by sending 
$a'$ to $\iota_i\circ a'\circ \iota_i^{-1}$ for $a'\in\uA'(T_i)$ and $T_i$ an $S_i$-scheme.
Then clearly $\iota_i$ induces an isomorphism $\ol{\iota}_i: (\uA'\backslash\uX')_{S_i}\to(\uA\backslash\uX)_{S_i}$.

One checks that $p\circ \iota_i=\pi_{S_i}$.
Hence $\ol{\iota}_i\circ p'=\pi_{S_i}$ where $p'$ is the projection from $\uX'$ to $\uA'\backslash\uX'$.
This implies that $\pi(x')=\pi(a'\cdot x')$ for $x'\in \uX'(T)$, $a'\in\uA'(T)$ and for arbitrary $S$-scheme $T$.
Hence $\pi$ induces a map $\ol{\pi}$ from $\uA'\backslash\uX'$ to $\uA\backslash \uX$.
On each $S_i$,  we have $\ol{\pi}_{S_i}=\ol{\iota}_i$. As $\ol{\iota}_i$ is an isomorphism,
$\ol{\pi}$ is an isomorphism between $\uA'\backslash\uX'$ and $\uA\backslash\uX$.

\smallskip

\noindent (4) Let $X'$ be an $S$-form of $X$ as $\uG$-spaces.
Then the sheaf of isomorphisms $\uIsom_\uG(X,X')$ is a right $\uA$-torsor and
the contracted product $\uIsom_{\uG}({\uX,\uX'})\land^\uA\uX$ 
(as defined in \cite[\S III.1.3]{Gir}) is canonically isomorphic to $\uX'$.

Conversely given a right $\uA$-torsor $\uF$, let $\uX'$  be the twist of $\uX$ by $\uF$.
According to the proof of  (3) , the twist $\uX'$ is a $G$-homogeneous space and is isomorphic to $\uX$ $fppf$-locally.
Hence $\uX'$ is an $S$-form of $\uX$.
We claim that the sheaf of isomorphisms $\uIsom(X, X')$ is isomorphic to $\uF$.

Note that  $\uA\land^{\uA} \uX$ is canonically isomorphic to $\uX$ and 
the $\uA$-action on $\uX$ corresponds to the $\uA$-action on the $\uA$-factor of $\uA\land^{\uA} \uX$.
Regard $\uA$ as a right $\uA$-torsor.
Then there is a natural map $\iota$ between $\uIsom_{\uA}(\uA,\uF)$ and $\uIsom_{\uG}(\uX,\uX')=\uIsom_{\uG}(\uA\land^\uA\uX,\uF\land^\uA\uX')$.
To be precise, for an $S$-scheme $S'$ and $\phi\in\uIsom_\uA(\uA,\uF)(S')$, $\iota(\phi)([(a,x)])=[(\phi(a),x)]$ for any $S'$-scheme $S''$ and $[(a,x)]\in (\uA\land^\uA\uX)(S'')$.

The sheaf of automorphisms of $\uA$ (as a right $\uA$-torsor) is $\uA$ itself which acts on the left of $\uA$.
As $\uIsom_{\uA}(\uA,\uF)$ is a right $\uAut_\uA(\uA)$-torsor, it is  a right $\uA$-torsor. 
It is easy to check that  $\iota$ is compatible with the $\uA$-action. 
Since $\uIsom_{\uA}(\uA,\uF)$ and $\uIsom_{\uG}(\uX,\uX')$ are both right $\uA$-torsors, and $\iota$ is compatible with the $\uA$-action, $\iota$ is an isomorphism.

There is a natural map $i$ from $\uIsom_\uA(\uA,\uF)$ to $\uF$ that sends $\phi\in \uIsom_\uA(\uA,\uF)(S')$ to $\phi(1_\uA)\in\uF(S')$ for all $S$-schemes $S'$.
Since $i$ is compatible the right $\uA$-action, $i$ is an isomorphism between right $\uA$-torsors, and $i\circ\iota^{-1}$ gives the desired isomorphism between $\uIsom_\uG(\uX,\uX')$ and $\uF$. 
Therefore the map that sends an $S$-form $\uX'$ to the right $\uA$-torsor $\uIsom_{\uG}(\uX,\uX')$  is a bijection.
\end{proof}

\section{Embedding functors}
Let $S$ be a scheme. 
For a point $s\in S$, let $\kappa(s)$ be the residue field of $s$
and $\ol{\kappa(s)}$ be the algebraic closure of $\kappa(s)$. Let
$\ol{s}$ be the scheme $\Spec(\ol{\kappa(s)})$. 

\subsection{Twisted root data and Weyl groups}

Let $T$ be an $S$-torus and denote by $\cM$
its sheaf of characters (for the fppf topology). We denote by $(-,  - )$ the 
canonical pairing $\cM^\vee \times \cM \to \ZZ_S$
We remind the reader of the following definition.
\cite[XXII.1.9]{SGA3}. A {\it twisted root datum} $\Psi= (\cM, \cM^\vee, \cR, \cR^\vee)$
with respect to $T$ consists of finite $S$-subschemes $\cR \subset  \cM_S$
and  $\cR^\vee \subset  \cM^\vee_S$ together with an isomorphism $\cR\simlgr \cR^\vee$
denoted by $\alpha \mapsto \alpha^\vee$
satisfying the following rules:

\smallskip

 (DR 1) For each $S$--scheme $S'$ and each  $\alpha \in \cR(S')$, 
 we have $(\alpha^\vee , \alpha) = 2$;

\smallskip

(DR 2) For each $S$--scheme $S'$ and each  $\alpha, \beta \in R(S')$, 
we have  $\alpha - (\beta^\vee, \alpha) \beta \in \cR(S')$,
 $\alpha^\vee - (\alpha^\vee, \beta) \beta^\vee \in \cR^\vee(S')$.

\smallskip

\begin{sexample} (Split root datum) {\rm
Let $M$ be a $\bZ$-lattice and $M^\vee$ be its dual lattice.
Let $R$ and $R^\vee$ be  finite subsets of $M$  and $M^{\vee}$ respectively.
Suppose that  $(M,M^\vee,R, R^\vee)$  satisfy the root data axioms. 
Denote by $\uM_S$ (resp.\ $\uR_S$) the constant sheaf on $S$ associated to $M$ (resp.\ $R$).  
Then $(\uM_{S},\uM_{S}^\vee,\uR_{S}, \uR_{S}^\vee)$ is a twisted
root datum for  the split torus $D_S(M)$.
It  is called the split twisted root datum associated to 
the root data $(M,M^\vee,R, R^\vee)$.
}
\end{sexample}

 Let $\Psi= (\cM, \cM^\vee, \cR, \cR^\vee)$ be a twisted root datum with respect to 
$T$. According to \cite[Prop.\ 6.1.3.(1)]{Gi2015},
$S$ admits an \'etale cover $(S_i)_{i \in I}$ such that 
$\Psi_{S_i}$ is split for each $i \in I$.

We say that the twisted root datum $\Psi$ is of \emph{type} $(M,M^\vee,R, R^\vee)$ at $s\in S$ if $\Psi_{\ol{s}}\simeq (\uM_{\ol{s}},\uM_{\ol{s}}^\vee,\uR_{\ol{s}}, \uR_{\ol{s}}^\vee)$.
The quoted result implies that the type is a locally constant function on $S$
and also that the maps $\cR \to \cM_S$,  $\cR^\vee \to \cM^\vee_S$
are clopen immersions.

We denote by $W(\Psi)$ the Weyl $S$-group scheme 
of the twisted root datum $\Psi$.
If $\Psi$ is split of constant type, $\Psi$ can be written as $(\uM_{S},\uM_{S}^\vee,\uR_{S}, \uR_{S}^\vee)$ for some root datum $(M,M^\vee,R, R^\vee)$.
In this case, we have $W(\Psi) = W_S$
where $W$  is the (abstract) Weyl group  
of the root datum $(M,M^\vee,R, R^\vee)$;
we recall that $W$ is 
a finite group generated by the reflections $s_\alpha$ defined by
\[s_\alpha(x)=x-(\alpha^{\vee},x)\alpha, \mbox { for $\alpha\in R$ and $x\in M$}.\]
For $\alpha\in R$, the reflection $s_\alpha$ induces an automorphism $\tilde{s}_\alpha$ on $T$ by

\[\tilde{s}_\alpha(t)=t(\alpha^{\vee}(\alpha(t)))^{-1}\, \mbox { for any $S$-scheme $S'$, $\alpha\in R$ and $t\in T(S')$}.\]
(See \cite{SGA3} Exp.\ XXII, 3.3.)

The map sending  $s_{\alpha}$ to $\tilde{s}_\alpha$ defines an action of $W(\Psi)^{op}$ on $T$.
For $w\in W(\Psi)(S)$, we denote by $\tilde{w}$ its image in $\uAut(T)(S).$
In general, $W(\Psi)$  is a finite \'etale group scheme over $S$, which is regarded as a subgroup scheme of $\uAut(\cM)$.

\begin{sremark} \label{rem_product} (Product of twisted root data)
{\rm  We have a direct product for root data \cite[XXI.6.4]{SGA3} but strangely the analogous notion for twisted root data is not defined.
Let $T_1, T_2, \dots, T_d$ be $S$--tori with respective character sheaves
$\cM_i$. We put ${T=T_1 \times_S \dots \times_S T_d}$, 
and denote by $p_i: T \to T_i$ the projection on the factor $i$
and $q_i: T_i \to T$;
its character sheaf is $\cM=\cM_1 \oplus \dots \oplus \cM_d$.
Let $\Psi_i= (\cM_i, \cM_i^\vee, \cR_i, \cR_i^\vee)$ be 
a twisted root data over $S$ ($i=1,...,d$).
We define $\Psi_1 \times \dots \times \Psi_d=\bigl(\cM ,\cM^\vee, \cR, \cR^\vee)$
with $\cR(X)= \sqcup_{i=1}^d p_i^*(\cR_i(X) ) \subset \cM(X) $,
 $\cR^\vee(X)= \sqcup_{i=1}^d q_{i,*}(\cR^\vee_i(X)) \subset \cM^\vee(X) $ 
for each $S$--scheme $X$. This is twisted root data for
the torus $T$; its  Weyl group $S$-scheme of the product is the product of Weyl groups
and the type function of  the product is the product of the type functions 
(as for root data \cite[XXI.6.4.2]{SGA3}). The automorphism group is more involved and we will  discuss it in Lemma \ref{lem_automorphism_Weil}, see also  Remark \ref{rem_automorphism_Weil}.
We have to pay attention that the 
Dynkin $S$--diagram as the product is not the product of the Dynkin $S$-diagrams;
indeed we have an indentification  $\uDyn(\Psi_1 \times \dots \times \Psi_d)
= \uDyn(\Psi_1) \sqcup   \dots \sqcup \uDyn( \Psi_d)$.}

\end{sremark}

The next fact is  not surprizingly in the literature.

\begin{slemma} \label{lem_descent}
 Let $(S_i)_{i \in I}$ be a fpqc cover of $S$ and 
let $(\Psi_i/S_i)_{i \in I}$ be twisted root data
equipped with a descent datum. Then this descent data is effective.  
\end{slemma}

\begin{proof}
 We write $\Psi_i= (\cM_i, \cM_i^\vee, \cR_i, \cR_i^\vee)$ be 
a twisted root data over $S_i$ where $\cM_i$ is the character group of an $S_i$--torus
$T$. By fpqc descent for affine schemes \cite[Tag 0245]{St}, 
the $T_i$'s (resp.\ $\cR_i$, $\cR_i^\vee$) descend to an $S$--torus $T$
(resp.\ to affine  $S$-schemes $\cR$, $\cR^\vee$).
In view of the permanence properties \cite[prop.\ 14.5.3.(7)]{GW},
 $\cR$ and  $\cR^\vee$ are finite $S$--schemes. Furthermore the 
 clopen immersions $\cR_i \to \cM_i$'s (resp.\ $\cR_i^\vee \to \cM_i^\vee$'s, resp.\
 the isomorphism $\cR_i \simlgr  \cR_i^\vee$)
 descend  to  a clopen immersion $\cR \to \cM$ 
 (resp.\ $\cR^\vee \to \cM^\vee$, resp.\ an isomorphism $\cR \simlgr \cR^\vee)$. 
The rules (DR1) and (DR2) for the  $\Psi_i$'s
imply that (DR1) and (DR2) are fulfilled by $\Psi$ as well.
Thus $\Psi$ is a twisted root datum.

\end{proof}

Let $G$ be a reductive group scheme over $S$.
Suppose that $G$ has a maximal torus $T$ over $S$.
We denote by $\Phi(G,T)$ the twisted root datum of $G$ with respect to $T$.
Let $\uN_G(T)$  the normalizer of $T$ in $G$. The conjugation action of $\uN_G(T)/T$ on $T$ gives an isomorphic from $\uN_G(T)/T$ to $W(\Phi(G,T))^{op}$. (cf. \cite{SGA3} Exp.\ XXII, 3.1-3.4)

\medskip

\subsection{The Orientation}
Let $\Psi_1$, $\Psi_2$ be two twisted root data. Suppose that $\Psi_1$ and $\Psi_2$ 
are of the same type at each $s\in S$. 
Let $\uIsom(\Psi_1,\Psi_2)$ be the
sheaf of isomorphisms between $\Psi_1$ and $\Psi_2$. Then
$\uIsom(\Psi_1,\Psi_2)$ is a right principal homogeneous space
of $\uAut(\Psi_1)$ and a left principal homogeneous of
$\uAut(\Psi_2)$. 
Define
\[\uIsomext(\Psi_1,\Psi_2)=W(\Psi_2)\backslash\uIsom(\Psi_1,\Psi_2).\]

Write the twisted root datum $\Psi_i$ as $(\cM_i,\cM_i^\vee, \cR_i, \cR_i^{\vee}).$
For an $S$-scheme $S'$ and $f\in \uIsom(\Psi_1,\Psi_2)(S')$, $f$ induces an isomorphism between the Weyl groups $W(\Psi_{1,S'})$ and $W(\Psi_{2,S'})$.
Namely for any $S'$-scheme $S''$ and any root $\alpha$ of $\cR_{1}(S'')$, $f$ sends the reflection $s_\alpha$ to $s_{f(\alpha)}$ and   $f\circ s_{\alpha}=s_{f(\alpha)}\circ f.$
Hence  $\uIsomext(\Psi_1,\Psi_2)$ is canonically isomorphic to $\uIsom(\Psi_1,\Psi_2)/W(\Psi_1).$
As a consequence, the natural isomorphism between sheaves $\uIsom(\Psi_1,\Psi_2)$ and $\uIsom(\Psi_2,\Psi_1)$, which sends $f\in \uIsom(\Psi_1,\Psi_2)(S')$ to $f^{-1}$, gives an isomorphism between $\uIsomext(\Psi_1,\Psi_2)$ and $\uIsomext(\Psi_2,\Psi_1).$

Suppose that $\uIsomext(\Psi_1,\Psi_2)(S)$ is nonempty.
For $u\in\uIsomext(\Psi_1,\Psi_2)(S)$, we define $\uIsomint_u(\Psi_1,\Psi_2)$ to be the fiber of $\uIsom(\Psi_1,\Psi_2)\to\uIsomext(\Psi_1,\Psi_2)$ at $u$.

Let $G$ be a reductive group scheme  over $S$. 
The \emph{type} of
$G$ at $s$ is the isomorphism class of the root datum of $G_{\ol{s}}$ with respect to its maximal torus. (ref.~\cite{SGA3},
Exp.\ XXII, Def.\ 2.6.1, 2.7)

Suppose that $G$ has a maximal torus $T$. We define $\uIsomext(G,\Psi)$ to be  $\uIsomext(\Phi(G,T),\Psi)$.
Let  $T^\sharp$ be another maximal torus of $G$ over $S$. For an $S$-scheme $S'$, every element of the transporter $\uTransp_G(T^\sharp,T)(S')$ gives an isomorphism from $T^\sharp_{S'}$ to $T_{S}$ via conjugation, which in turns
gives an isomorphism from $\Phi(G,T)$ to $\Phi(G,T^\sharp)$. 
As $\uTransp_G(T^\sharp,T)$ is a right homogeneous space  under $\uN_G(T^\sharp)$, through the identification of $\uN_G(T^\sharp)/T^\sharp$ with $W(\Phi(G,T^\sharp))^{op}$, we see that there is 
a natural morphism  \[\uIsom(\Phi(G,T^\sharp),\Psi)\land^{\uN_G(T^\sharp)^{op}}\uTransp_G(T^\sharp,T)\to \uIsom(\Phi(G,T),\Psi).\]
This gives a canonical isomorphism from $\uIsomext(\Phi(G,T^\sharp),\Psi)$ to $\uIsomext(\Phi(G,T),\Psi)$.
 Hence $\uIsomext(G,\Psi)$ is well-defined. For those $G$ without maximal torus over $S$, we define $\uIsomext(G,\Psi)$
 by descent (ref. \cite[\S 1.2.1]{L1}).

We define $\uIsomext(\Psi,G)$ in a similar way.
As there is a canonical isomorphism between $\uIsomext(\Phi(G,T),\Psi)$ and $\uIsomext(\Psi,\Phi(G,T))$, 
the two sheaves $\uIsomext(G,\Psi)$ and $\uIsomext(\Psi,G)$ are canonically isomorphic.

An \emph{orientation} of $G$ with respect to $\Psi$ is an element of $\uIsomext(G,\Psi)(S).$

\smallskip

A twisted root datum $\Psi$ is \emph{admissible} for $G$ if at each $s\in S$,  the type of $G_{\ol{s}}$ is the same as the type of   $\Psi_{\ol s}$,
 and $\uIsomext(G,\Psi)(S) \not= \emptyset$. The \emph{admissibility} condition is a necessary condition for the existence of a maximal torus $T$ of $G$ such that $\Phi(G,T)$ is isomorphic to $\Psi$. 
 
Let $G'$ be an $S$-form of $G$.
Let $\uIsom(G,G')$ be the sheaf of group isomorphisms between $G$ and $G'$.
Note that $G$ acts on itself by conjugation. Thus we can define the right quotient of $\uIsom(G, G')$ by the adjoint quotient $G_{ad}$ of $G$ and 
denote  $\uIsom(G, G')/G_{ad}$ by $\uIsomext(G,G')$.
  
  \smallskip

 Let $\Psi_1$, $\Psi_2$ and $\Psi_3$ be twisted root data over $S$. Suppose that
all of them are of the same type at each geometric fibre of $S$. Then
we have the following morphism \label{pairing1}:
\begin{equation}
\uIsomext(\Psi_1,\Psi_2)\times\uIsomext(\Psi_2,\Psi_3)\to\uIsomext(\Psi_1,\Psi_3),
\end{equation}
that comes from the composition of isomorphisms 
$$\uIsom(\Psi_1,\Psi_2)\times\uIsom(\Psi_2,\Psi_3)\to\uIsom(\Psi_1,\Psi_3).$$

Suppose $G$ and $\Psi_1$ are of the same type at each geometric fiber of $S$.
Similarly we have the following pairings coming from the composition of isomorphisms between root data.

\begin{proposition}\label{pairing}
We have the following pairings
\begin{enumerate}
 \item $\uIsomext(G,\Psi_1)\times\uIsomext(\Psi_1,\Psi_2)\to\uIsomext(G,\Psi_2).$
 \item  $\uIsomext(G,\Psi_1)\times\uIsomext(G,\Psi_2)\to\uIsomext(\Psi_1,\Psi_2).$ 
 \item $\uIsomext(G,\Psi_1)\times\uIsomext(G',\Psi_1)\to\uIsomext(G,G').$  
 \item $\uIsomext(G',G) \times \uIsomext(G,\Psi)\to \uIsomext(G’,\Psi)$.

 \end{enumerate}
 \end{proposition}
\begin{proof}
The first two assertions come from the definition of $\uIsomext(G,\Psi)$ and the composition of isomorphisms between root data.
For (3) and (4), we notice that if $G$ and $G'$ have maximal tori $T$ and $T'$ respectively, then $\uIsomext(\Phi(G,T),\Phi(G',T'))$ is representable and isomorphic to $\uIsomext(G,G')$.
As reductive groups have maximal tori  \'etale locally, assertions (3) and (4) can be deduced from descent. 
We refer to \cite[Prop.\ 1.4]{L1} for more details.
\end{proof}

\smallskip

Another complement is the following criterion of 
reduction of torsors.
Let $G$ be a reductive group scheme over $S$ with a maximal torus $T$.
Let $E$ be right $fppf$ $G$-torsor over $S$.
Consider the action of $G$ on itself by conjugation and  set $G^\sharp =E\land^G G$. 
As $G^\sharp$ is an inner twisted form of $G$, there is a canonical element $c\in\uIsomext(G,G^\sharp)(S)$.
Let  $\Psi$ and $\Psi^\sharp$ be admissible
twisted root data of $G$  and $G^\sharp$ respectively.
For any such pair $(\Psi,\Psi^\sharp)$, the sheaf  $\uIsomext(G,G^\sharp)$ is canonically isomorphic to $\uIsomext(\Psi,\Psi^\sharp)$. Thus we can regard $c$ as an element in $\uIsomext(\Psi,\Psi^\sharp)(S)$. 

\begin{proposition}\label{prop_new}
Let $E$, $T$, $G$ and $G^\sharp$ be as above 
and consider the twisted root datum $\Phi=\Phi(G,T)$.
Then the following assertions are equivalent:

\smallskip

(i) $E$ admits a reduction to $T$, 
i.e.\ there is a right  $T$-torsor $F$ such that $E\simeq F\land^T G$ as $G$-torsors;

\smallskip

(ii) $G^\sharp$ admits  a maximal $S$--torus $T^\sharp$ satisfying the following property:

\smallskip
the set $\uIsomint_c(\Phi,\Phi^\sharp)(S)$ is nonempty, 
with the notation $\Phi^\sharp=\Phi(G^\sharp,T^\sharp)$.

\end{proposition}

\begin{proof}
Let $\uIsom_G(G,E)$ be the fppf-sheaf of isomorphisms of right $G$-torsors. Denote by $1_G$  the neutral element of $G(S)$.
There is an isomorphism $\iota$ from $\uIsom_G(G,E)$ to $E$ which sends 
$f\in\uIsom_G(G,E)(S')$ to $f(1_G)$ for any $S$-scheme $S'$. 
Regard $G$ as a right $G$-torsor,  the left multiplication of $G$ on itself forms the automorphism group $\uAut_G(G)$.
Thus $\uIsom_G(G,E)$ is a right $G$-torsor and $\iota$ is an isomorphism of right $G$-torsors.

Denote by $\uIsomint_c(G,G^\sharp)$ the fiber of $\uIsom(G,G^\sharp)\to \uIsomext(G,G^\sharp)$ at $c$.
Clearly an isomorphism between the $G$-torsors $G$ and $E$ gives an isomorphism between reductive groups $G$ and $G^\sharp$ compatible with the orientation $c$. Hence  there is a natural surjective morphism of sheaves from $\uIsom_G(G,E)$ to $\uIsomint_c(G,G^\sharp)$.

\smallskip

\noindent $(ii) \Longrightarrow (i)$. We are given a maximal $S$--torus $T^\sharp$ of $G^\sharp$
such that $\uIsomint_c(\Phi,\Phi^\sharp)(S)\neq\emptyset$.
Next consider the sheaf of isomorphisms $\uIsomint_c((G,T),(G^\sharp,T^\sharp))$, which consists of sections in $\uIsomint_c(G,G^\sharp)$   that take $T$ to $T^\sharp$.
This is a subsheaf of $\uIsomint_c(G,G^\sharp)$. 
Set \[\cF=\uIsomint_c((G,T),(G^\sharp,T^\sharp))\underset{\uIsomint_c(G,G^\sharp)}{\times}\uIsom_G(G,E).\]
We can regard $\cF$ as a subsheaf of $\uIsom_G(G,E)$ through the second projection $p_2$.  
Note that $\uN_G(T)$ acts diagonally on the right of both factors of $\cF$.
Clearly $p_2$ is an $\uN_G(T)$-equivariant morphism. 
Let $\pi:\cF\to \uIsomint_c(\Phi,\Phi^\sharp)$ be the projection from $\cF$ to $\uIsomint_c((G,T),(G^\sharp,T^\sharp))$ composed with the natural morphism from $\uIsomint_c((G,T),(G^\sharp,T^\sharp))$ to $\uIsomint_c(\Phi,\Phi^\sharp).$

Let $f\in\uIsomint_c(\Phi,\Phi^\sharp)(S)$. Then the fiber of $\pi$ at $f$ is a right $T$-torsor and we denote it by $F$.
Note that $F$ embeds in $E$ via $\iota\circ p_2$.
Hence we have $E= F \wedge^T G$.

\smallskip

\noindent $(ii) \Longrightarrow (i)$. We suppose $E=F\land^T G$ for some $T$-torsor $F$.
Then $G^\sharp=E\land^G G=F\land^T G$, and $T^\sharp=F\land^T T$ is a maximal torus of $G^\sharp$.
It remains to show that $\uIsomint_c(\Phi,\Phi^\sharp)(S)\neq\emptyset$.

Let $S'\to S$ be a $fppf$-cover that trivializes $F$.
We choose a trivialization $f:T_{S'}\xrightarrow{\sim} F_{S'}$.
Then $f$ induces an isomorphism between $(G_{S'},T_{S'})$ and $(G^\sharp_{S'},T^\sharp_{S'})$ with respect to the orientation $c$, 
which in turn gives an isomorphism $f^\diamond \in\uIsomint_c(\Phi,\Phi^\sharp)(S')$.

Set $S''=S'\times_S S'$.
Let $r_1$ and $r_2$ be the projections of $S''$  to its first and second factors respectively. 
Then $(r_2^\ast f)^{-1}\circ(r_1^\ast f)$ is an automorphism of $T_{S''}$ (as a $T_{S''}$-torsor),
which  corresponds to  the left multiplication by some element $t\in T(S'')$.
As $T$ acts trivially on itself by conjugation, $(r_2^\ast f)^{-1}\circ(r_1^\ast f)$ induces trivial automorphism on $\Psi_{S''}.$
Hence $r_2^\ast (f^\diamond )=r_1^\ast (f^\diamond )$ and $f^\diamond $ descends to an element in
$\uIsomint_c(\Phi,\Phi^\sharp)(S)$. Thus $\uIsomint_c(\Phi,\Phi^\sharp)(S) \not = \emptyset$ 
as desired.
\end{proof}

\smallskip

Denote by $G//G$ the adjoint quotient of $G$ in the GIT framework, 
its formation commutes with arbitrary base change \cite[\S 4]{L2}.
Suppose that $G$ has a maximal torus $T$.
Let $\uN_G(T)$ be the normalizer of $T$ in $G$ and set $W_G(T)=\uN_G(T)/T$. 
The conjugation action of $\uN_G(T)$ on $T$ induces an action of the Weyl group $W_G(T)$ on $T$.
Let $T/W_G(T)$ be the quotient sheaf of $T$ by the action of Weyl group $W_G(T)$.
Then $T/W_G(T)$ is isomorphic to $G//G$ (\cite{L2} Thm.\ 4.1).

Consider the quotient map (of $\fppf$-sheaves) $\pi: T\to T/W_G(T)$.
Let $x\in (G//G)(S)$. As $T/W_G(T)$ and $G//G$ are isomorphic, we can regard  $x$ as an element in $(T/W_G(T))(S).$
Suppose that there is a semisimple regular element $t_1\in G(S)$ whose image in $(G//G)(S)$ is $x$.
Let $T_1$ be the centralizer $\uC_G(t_1)$ so that $T_1$ is also a maximal torus.
Let $\Phi$ (resp.\ $\Phi_1$) be the root datum associated to $T$ (resp.\ $T_1$). 
Then the natural inclusion of $T$ into $G$ gives an orientation $v \in\uIsomext(G, \Phi)(S)$.
Similarly, we get an orientation $v_1\in \uIsomext(G,\Phi_1)(S)$ from the natural inclusion of $T_1$ into $G$.
Then by  \ref{pairing} (2) we get an orientation
$u\in \uIsomext (\Phi,\Phi_1)(S)$ coming from $v\cdot v_1$.
Then we have the following proposition.

\begin{proposition}\label{W-torsor}
The left $W_G(T)$-torsor $\pi^{-1}(x)$ is isomorphic to $\uIsomint_u(\Phi,\Phi_1)$.
\end{proposition} 

\begin{proof}
For an $S$-scheme $S'$, we denote the image of $t_1$  in $G(S')$ by $t_{1,S'}$. For $t\in \pi^{-1}(x)(S’)$, we denote by $\uTransp(t,t_{1,S'})$ the transporter from $t$ to $t_{1,S'}$.
Both $t$ and $t_{1,S'}$ are mapped to the same element in $(G//G)(S')$, and hence $\uTransp(t,t_{1,S'})\to \bullet$ is surjective.
Since $t_1$ is a semisimple regular element, so is $t$.
As $t$ lies in $T$,  the centralizer $\uC_G(t)$ of $t$ is the torus $T_{S’}$ and $\uTransp(t,t_{1,S'})$ is a right $T_{S’}$-torsor.

The conjugation action of $T$ on $\Phi=\Phi(G,T)$ is trivial.
Therefore the canonical morphism $$\uTransp(t,t_{1,S'})\land^{T_{S'}}\Phi_{S'}\to \Phi_{1,S'}$$ defines an element in $\uIsomint_{u}(\Phi,\Phi_1)(S’) $.
In this way, we get a map $$i:\pi^{-1}(x)\to\uIsomint_{u}(\Psi,\Psi_1).$$
To be precise, write the root datum $\Phi$ as $(\cM,\cM^\vee,\cR,\cR^\vee)$.
For an $S'$-scheme $S''$ and $h\in\uTransp(t,t_{1,S'})(S'')$, $i(t)(m)=m\circ int (h^{-1})$ for all $m\in \cM_{S^{\prime\prime}}=\Hom_{S^{\prime\prime}}(T_{S^{\prime\prime}}, \G_{m,S^{\prime\prime}})$.
Note that $i(t)$ is independant of the choice of $h$ since $T$ acts trivially on $\cM$. 

As $\uIsomint_{u}(\Psi,\Psi_1)$ is a right $W(\Psi)$-torsor and $W_G(T)\simeq W(\Psi)^{op}$, we can regard it as left $W_G(T)$-torsor.
We write down explicitly this action.
For $w\in W_G(T)_{S'}$, let $n_w$ be a lift of $w$ in $N(T)(S'')$ for some $S'$-scheme $S''$.  
As $T$ acts on $\cM$ trivially,  $int (n^{-1}_w)$ induces an automorphism on $\cM_{S'}$.
 For $f\in \uIsomint_{u}(\Phi,\Phi_1)(S’)$,
 $(w\cdot f)(m):=f(m\circ int (n_w))$ for $m\in \cM_{S’}=\Hom_{S’-grp}(T_{S’},\G_{m,S’})$.

Since $\pi^{-1}(x)$ and $\uIsomint_u(\Phi,\Phi_1)$ are both $W_G(T)$-torsors, to check $i$ is an isomorphism, it suffices to check the map $i$ is $W_G(T)$-equivariant.
Let $w\in W_G(T)(S')$ and $n_w$ be a lift of $w$ in $\uN_G(T)(S'')$ for some $S'$-scheme $S''$.  
For  $h\in \uTransp(t,t_{1,S'})(S'')$,  $hn^{-1}_w\in  \uTransp(w\cdot t,t_{1,S'})(S'')$. 
Thus
\begin{equation}
\begin{aligned} 
i(w\cdot t)(m) & =m\circ int((hn^{-1}_w)^{-1})\\
&=m\circ int(n_w)\circ int(h^{-1})\\
&=i(t)(m\circ int(n_w))\\
 &=(w\cdot i(t))(m). 
\end{aligned}
\end{equation}
Hence $i$ is an isomorphism between $W_G(T)$-torsors.
\end{proof}

\subsection{Embedding functors}
Let $G$ be a reductive group scheme over $S$ and  \break $\Psi= (\cM, \cM^\vee, \cR, \cR^\vee)$ be an admissible root datum for $G$.
Let $T$ be the $S$-torus with sheaf of characters $\cM$.

Suppose there is an embedding $f$ of $T$ into $G$ as algebraic group schemes.
Then $f$ induces an isomorphism from the character group of $f(T)$ to $\cM$.
We call this map $f^\sharp$.
If $f^\sharp$ induces an isomorphism between $\Phi(G,f(T))$ and $\Psi$, then we say $f$ is an embedding with respect to the twisted root datum $\Psi$.

Define the \emph{embedding
functor} $ \gE(G,\Psi)$ as follows: for each $S$-scheme $S'$,  \[ \gE(G,\Psi)(S')=\left\{\begin{array}{l}
\mbox{$f:T_{S'}\hookrightarrow G_{S'}$}\left|\begin{array}{l}\mbox{$
f$ is both a closed
immersion and a group }\\
 \mbox{homomorphism that induces an isomorphism
}\\
\mbox{$f^{\sharp}:\Phi(G_{S'},f(T_{S'}))\xrightarrow{\sim}\Psi_{S'}$
}\\

\end{array}\right.\end{array}\right\}\]

Note that the fact that $\Psi$ is admissible for $G$ ensures that $\gE(G,\Psi)$ is not empty and allows us to fix an orientation $v \in \uIsomext(G,\Psi)(S)$.

The
\emph{oriented embedding functor} $\gE(G,\Psi,v)$  over $S$ is defined by

\[\gE(G,\Psi,v)(S')=\left\{\begin{array}{l}\mbox{$f:T_{S'}\hookrightarrow G_{S'}$}\left |
\begin{array}{l}\mbox{$f\in\gE(G,\Psi)(S')$,
and the image of $f^{\sharp}$ } \\
\mbox{ in $\uIsomext(G,\Psi)(S')$ is
$v$.}\end{array}\right.\end{array}\right\}\]

Since any two maximal tori of $G$ are \'etale locally conjugated by $G$,
the oriented embedding functor $\gE(G,\Psi,v)$ is representable by an affine $S$--scheme that  is a left homogeneous $G$-space (\cite[1.6]{L1}).  
For an $S$-scheme $S'$ and an element $f\in \gE(G,\Psi,v)(S')$, the stabilizer of $f$ in $G_{S'}$ is $f(T_{S'})$, which is a maximal torus of $G_{S'}$.
Recall that for $w\in W(\Psi)(S')$, $w$ induces an automorphism on $T_{S'}$ and we denote by $\tilde{w}$ its image in $\uAut(T)(S')$ (\S 3.1).
This gives a group monomorphism from $W(\Psi)^{op}$ to $\uAut(T)$.
The $S$--scheme $\gE(G,\Psi,v)$ admits a right  action of
$W(\Psi)^{op}$ and the quotient sheaf $\gE(G,\Psi,v)/W(\Psi)^{op}$ identifies with the
$S$-scheme $\cT_G$ of maximal tori of $G$. (See \cite[\S 1.1 and \S1.2]{L1} for more details.)

\begin{sremark}{\rm
In \cite[\S 1.2.2]{L1} we define the oriented embedding functor $\gE(G,\Psi,v')$ by orientation $v'\in \uIsomext(\Psi,G)(S)$ instead of $v\in \uIsomext(G,\Psi)(S)$ here.
However when $v'$ is the image of $v$ under the canonical isomorphism $\uIsomext(G,\Psi)\to\uIsomext(\Psi,G)$, these two definitions are clearly equivalent.
}
\end{sremark}

\begin{slemma} \label{lem_compatibilities}
Keep the notation as above.
\smallskip

 (1) We have an isomorphism $W(\Psi) \simlgr \uAut_G(\gE(G,\Psi,v))$ as group schemes.

\smallskip

(2) Let $\Psi'$ be another admissible root datum for $G$,
and $v' \in \Isomext(G,\Psi')(S)$. Let $u \in \uIsomext(\Psi,\Psi')(S)$ be the orientation  $v’ \cdot v$  by the pairing (\ref{pairing}). Under the identification of $W(\Psi)$ and $\uAut_G(\gE(G,\Psi,v))$ in (1), there is a natural isomorphism between
 $\uIsomint_u( \Psi, \Psi')$ and $\uIsom_G(\gE(G,\Psi,v), \gE(G,\Psi',v'))$ as right $W(\Psi)$--torsors.

\smallskip

(3)
Let $F$ be a right ${W}(\Psi)$-torsor and $\Psi'=F\land^{W(\Psi)}\Psi$.
Let $c$ be the canonical element of $\uIsomext(\Psi,\Psi')(S)$, and
 let $v_1\in \uIsomext(G, {\Psi'})$ be  $c\cdot v$
 by the pairing (\ref{pairing}). 
By identifying $W(\Psi)$ with $\uAut_G(\gE(G,\Psi,v))$, $F\land^{W(\Psi)}\gE(G,\Psi,v)$
is isomorphic to  $\gE(G, {\Psi'},v_1)$.

\smallskip

(4)  Let $E$ be a right $G$--torsor and denote by $G'$ the twist of $G$ by $E$ via inner automorphisms.
Let $c$ be the canonical element of $\uIsomext(G',G)(S)$.
Then the homogeneous $G'$-space  $E \land^{G} \gE(G,\Psi,v)$ is isomorphic
to   $\gE(G',\Psi,v')$ where $v' \in \uIsomext(G',\Psi)(S)$ is  $c\cdot v$ by the pairing  (\ref{pairing}).
\smallskip

\end{slemma}

\begin{proof}
(1) For an $S$-scheme $S'$ and $w\in W(\Psi)(S')$, we denote by $\tilde{w}$ the image of $w$ in $\uAut(T)(S')$   (see \S 3.1).
Define a morphism $\iota:W(\Psi)\to \uAut_G(\gE(G,\Psi,v))$ as follows.
For an $S$-scheme $S'$ and  $w\in W(\Psi)(S')$,
\[\iota(w)(x):=x\circ \tilde{w} \mbox{  for all $S'$-schemes $S''$,  and $x\in \gE(G,\Psi,v)(S'')$.}\]

For $w_1$ and $w_2\in W(\Psi)(S')$,
\begin{equation}
\begin{aligned}
\iota(w_1\circ w_2)(x)&=x\circ\widetilde{(w_1\circ w_2)}\\
&=x\circ\tilde w_2\circ\tilde w_1\\
&=(\iota(w_1)\circ\iota(w_2))(x).
\end{aligned}
\end{equation}
Hence $\iota$ is a group homomorphism. 

Clearly $\iota$ is monomorphism.
 To see that $\iota$ is an isomorphism, it suffices to check $\iota$ is an isomorphism of fppf-sheaves.
 Let $\{S_i\}$ be a $fppf$-cover of $S$ such that $G_{S_i}$ and $\Psi_{S_i}$ are both split.
 Then $\gE(G,\Psi,v)(S_i)$ is not empty.
 Choose $x_i\in\gE(G,\Psi,v)(S_i)$. By  Lemma \ref{lem_quotient},
 $\uAut_{G}(\gE(G,\Psi,v))(S_i)$ is isomorphic to $(\uN_G(x_i(T))/x_i(T))^{op}$, where $\uN_G(x_i(T))$ is the normalizer of $x_i(T)$ in $G$.
Since $\uN_{G}(x_i(T))/x_i(T)\simeq {W}(\Phi(G_{S_i},x_i(T_{S_i})))^{op}$,  the group $\uN_{G}(x_i(T))/x_i(T)$ is isomorphic to ${W}(\Psi_{S_i})^{op}.$
Thus $\uAut_{G}(\gE(G,\Psi,v))_{S_i}\simeq W(\Psi_{S_i})$.
As $\iota$ is a monomorphism and $\uAut_{G}(\gE(G,\Psi,v))_{S_i}$ is isomorphic to $W(\Psi_{S_i})$, $\iota$ is an isomorphism.

\smallskip

\noindent (2)
Write $\Psi'=(\cM',(\cM')^\vee, \cR',(\cR')^\vee)$ and denote by $T'$ the torus with sheaf of characters $\cM'$.
For an $S$-scheme $S'$ and $f \in\uIsomint_u(\Psi,\Psi')(S')$, we denote by $\tilde{f}$ the isomorphism from $T'_{S'}$ to $T_{S'}$ induced by $f$.
Define $\eta:\uIsomint_u(\Psi,\Psi')\to \uIsom_G(\gE(G,\Psi,v),\gE(G,\Psi',v'))$ by $\eta(f)(x)=x\circ \tilde{f} $,
for all $S'$-scheme $S''$, and $x\in\gE(G,\Psi,v)(S'')$.

Next we show that $\eta$ is compatible with the $W(\Psi)$-action under the identification of $W(\Psi)$ and $\uAut_G(\gE(G,\Psi,v))$.
Let $w\in W(\Psi)(S')$.
Then
\begin{equation}
\begin{aligned}
\eta(f\circ w)(x)&=x\circ (\widetilde{f\circ w})\\
&=x\circ \tilde{w}  \circ \tilde{f}\\
&=\eta(f)(x\circ\tilde{w})\\
&=\eta(f)( \iota(w)(x))\\
&=(\eta (f)\circ\iota(w))(x).
\end{aligned}
\end{equation}

This shows that $\eta $ is compatible with the $W(\Psi)$-action.
Since $\uIsomint_u(\Psi,\Psi')$ and $\uIsom_G(\gE(G,\Psi,v),\gE(G,\Psi',v'))$ are both right $W(\Psi)$-torsors, this implies that $\eta$ is an isomorphism.

\noindent (3)
Write $\Psi'$ as $(\cM',(\cM')^\vee,\cR',(\cR')^\vee)$ and let $T'$ be the torus with character group $\cM'$.
For an $S$-scheme $S'$ and $f\in F (S')$, we define $\varphi_f^\sharp:\cM_{S'}\to\cM'_{S'}$ as
$\varphi_f^\sharp(m)=(f,m)$ for all $S'$-scheme $S''$ and $m\in \cM_{S'}(S'')$.
Denote by $\varphi_f$ the group homomorphism from $T'$ to $T$ defined by $\varphi_f^\sharp$.

Define $\varsigma:F\times \gE(G,\Psi, v)\to\gE(G,\Psi',v')$ as follows.
For an $S$-scheme $S'$ and $(f,x)\in (F\times\gE(G,\Psi,v))(S')$, let
$\varsigma(f,x)=x\circ \varphi_f$. Clearly $x\circ \varphi_f$ is an embedding of $T'$ in $G$
with orientation $v'$.
Note that for $w\in W(\Psi)(S')$, we have
\begin{equation}
\begin{aligned}
 \varphi_{fw}^\sharp(m)&=(fw,m)\\
 &=(f,w^{-1}m)\\
 &=\varphi_f^\sharp(w^{-1}m).
\end{aligned}
\end{equation}
Therefore $\varphi_{fw}^\sharp=\varphi_f^\sharp\circ w^{-1}$ and
$\varphi_{fw}=\widetilde{w}^{-1}\circ\varphi_f$.
It follows that
\begin{equation}
\begin{aligned}
\varsigma(fw,x)&=x\circ \varphi_{fw}\\
&=x\circ \widetilde{w}^{-1}\circ \varphi_f\\
&=(w^{-1}\cdot x)\circ\varphi_f\\
&=\varsigma(f,w\cdot x).
\end{aligned}
\end{equation}
Hence $\varsigma$ induces a morphism from $F\land^{W(\Psi)}\gE(G,\Psi,v)$ to
$\gE(G,\Psi',v')$, that is clearly an isomorphism.

\smallskip

\noindent (4)
For an $S$-scheme $S'$ and $e\in E(S')$, we define $f_e: G_{S'}\to G'_{S'}$ by 
$f_e(g)=[e,g]$ for all $S'$-scheme $S''$ and $g\in G(S'')$. Then $f_e$ is an isomorphism between $G_{S'}$ and $G'_{S'}$.

For $y=(e,x)\in (E\times_S \gE(G,\Psi, v))(S')$,
we define \[\imath: E\times_S \gE(G,\Psi,v)\to \gE(G',\Psi,v')\] as
$\imath(y)=f_e\circ x$. It is clear that $\imath(y)\in\gE(G',\Psi)(S').$
As $x$ is an embedding of orientation $v$, by the definition of canonical orientation, the embedding  $\imath(y)$ is of orientation $v'$.
Since $\imath(e,x)=\imath(e\cdot g,g^{-1}\cdot x)$ for all $g\in G(S')$, $\imath$ induces an isomorhism from $E\land^G \gE(G,\Psi, v)$ to $\gE(G',\Psi,v')$.
\end{proof}

\begin{sremark}\label{rem_behaviour}{\rm
Let $G^{sc}$ be the simply connected cover of $G$
and let $G_{ad}$ be its adjoint quotient.
On the other hand let $\Psi^{sc}$ (resp.\ $\Psi_{ad}$)
be the simply connected (resp.\ adjoint) root datum associated to
$\Psi$.
Then we can attach to $(\Psi,v)$ an oriented
root datum $(\Psi^{sc},v^{sc})$ (resp.\ $(\Psi_{ad},v_{ad})$)
and isomorphisms
$$
\gE(G^{sc},\Psi^{sc},v^{sc}) \simlgr \gE(G,\Psi,v) \simlgr \gE(G_{ad},\Psi_{ad},v_{ad}).
$$
It implies that we can deal in practice with semisimple simply connected
(resp.\, adjoint) group schemes.
Another advantage is that we have isomorphisms of finite \'etale
group schemes
$$
\uIsomext(G^{sc}, \Psi^{sc} ) \simlgr  \uIsomext(G_{ad}, \Psi_{ad} )
\simlgr \uIsom( \uDyn(G), \uDyn(\Psi)).
$$
} (See \cite[Cor.\ 1.7]{L1}.)
\end{sremark}

\subsection{Split and quasi-split cases} \label{subsec_quasi_split}
We denote by $\gg=\Lie(G)$ the Lie algebra of $G$, this is 
an $G-\cO_S$--module.
We remind the reader of \S XXII.1,  XXIV.3.8 and 3.9 of \cite{SGA3}.

A {\it splitting} of $G$ is the data of a maximal split $S$--torus $T$ of $G$
together with a root  system $(M,R)$ in the sense of \cite[XIX 3.6]{SGA3}, 
an isomorphism $T \cong D_S(M)$ such that
\smallskip

(D1) $S$ is not empty and the roots  of $R$ (resp.\ the corresponding
coroots) of  $M$ (resp. $M^\vee$) are identified
with constant functions of $S$ in $M$ (resp. $M^\vee$).

\smallskip

(D2) For each $\alpha \in R$, the eigenspace $\gg_\alpha$ for 
the adjoint representation is a free invertible $\cO_S$--module.

\smallskip

The reductive $S$--group scheme $G$ together with a splitting
is called a {\it split reductive $S$--group scheme}.

A Killing couple for $G$ is a pair $(T,B)$ where 
$T$ is a maximal $S$--torus of $G$ and
$B$ is an $S$-Borel subgroup scheme of $G$ containing $T$.
Let $(T,B)$ be a Killing couple for $G$.
As in the field case, there exists a canonical monomorphism of 
$S$-sheaves $i: \uDyn(G) \to   \cM$  (where $\cM$ stands for the 
sheaf of characters of $T$) 
permitting to think about  $\uDyn(G)$ as the ``simple roots of $B$ relatively to $T$".
Next $\uDyn(G)$ is representable by the finite \'etale scheme 
$\bD=\Dyn(G)$.
We denote by $\alpha_\bD \in \cM(\bD)$ the universal root, that is, 
the image of $id_{\bD}$ by $i$. It gives rise to the associated eigenspace
$$\gg^\bD= \bigl( \gg \otimes_{\cO_S} \cO_{\bD} \bigr)^{\alpha_{\bD}}$$
 which is an invertible $\cO_{\bD}$-module and which
 is locally a direct summand of
  $\gg \otimes_{\cO_S} \cO_{\bD}$.
A {\it quasi-pinning} of $G$ is a triple $(T,B,X)$ where 
$(T,B)$ is a Killing couple as above and $X \in H^0(S, \gg^\bD)^\times$.
In other words, $X$ is a trivialization $\cO_{\bD} \simlgr \gg^\bD$
of the invertible $\cO_D$-module $\gg^\bD$.

We say that $G$ is {\it  quasi-split} if $G$ admits a quasi-pinning,
this is equivalent to require that $G$ admits a 
Killing couple if $S$ is a semilocal scheme \cite[XXIV.3.9.1]{SGA3}.
Those group schemes are quite easy to classify by relevant finite
\'etale covers, see \cite[XXIV.3.11]{SGA3}.
Furthermore a reductive $S$--group scheme
is a inner twisted $S$-form of a reductive quasi-split $S$--group
scheme \cite[\S 8.1]{Gi2015}.

\begin{remark} {\rm If $G$ is split with root system $(M,R)$,
 then   $\uDyn(G)$ identifies 
with the  finite constant 
finite scheme associated to  the Dynkin diagram $\Delta$  of $(M,R)$. 
In this case  $X$  is 
the data of generators $(X_\alpha)_{\alpha \in \Delta}$ of the root spaces $\gg_\alpha$ 
for $\alpha \in \Delta$.  \rm}
\end{remark}

\subsection{The main result}

\begin{stheorem}\label{thm_main}
Let $X$ be a separated $S$-scheme.
Assume that one of the following conditions is satisfied.

\smallskip

(i) $S$ is locally noetherian;

\smallskip

 (ii) $X$ is locally of finite type over $S$.

\smallskip

\noindent We assume that $X$ is a left homogeneous space under a reductive $S$--group 
scheme $G$
and that for each point $s \in S$, the stabilizer of $X_{\ol s}$ is a maximal torus  of $G_{\ol s}$.

\smallskip

(1) There exists a twisted root datum $\Psi= (\cM, \cM^\vee, \cR, \cR^\vee)$
that is admissible for $G$ and an orientation $v \in \uIsomext(G, \Psi)(S)$
such that $X$ is isomorphic to $\gE(G,\Psi,v)$ as left $G$-homogeneous space.

\smallskip

(2) The oriented root datum $(\Psi,v)$ in (1) is unique up to isomorphism.

\end{stheorem}

\begin{proof} We shall establish first that the local stabilizers
of $X$ are maximal tori.
Since $G$ is smooth over $S$ with connected fibers,
$X$ is smooth of finite presentation over $S$  according to \cite[prop.\ VI.1.2]{R}.

\smallskip

\noindent (1) Let $(S_i)_{i \in I}$ be a fppf cover of $S$ such that $X_{S_i}$ is $G_{S_i}$--isomorphic
to $G_{S_i}/H_i$ where $H_i$ is an $S_i$-subgroup scheme of $G_{S_i}$ (see \cite [IV 6.7.3]{SGA3}).
Since $H_i$ is the stabilizer of a point of $X(S_i)$ (which is $S_i$--separated),
$H_i$ is  a closed $S_i$--subgroup scheme of $G_{S_i}$ and is in particular affine.

Again by \cite[prop.\ VI.1.2]{R}, $H_i$ is flat locally of finite presentation over $S_i$
so is affine flat  of finite presentation over $S_i$.
As the geometric fibers of $H_i$ are tori, it follows that $H_i$ is an $S_i$--torus
\cite[X.4.9]{SGA3}. Therefore $H_i$ is a maximal $S_i$--torus of $G_{S_i}$ for each $i \in I$.

We denote by $G^q$ an inner quasi-split $S$--form of $G$.
Since $G$ is an inner form of $G^q$, there is a canonical orientation $c\in\uIsomext(G,G^q)(S)$.
Let $T^q$ be a maximal torus of $G^q$, $\Psi^q$ be the root datum  $\Phi(G^q,T^q)$ and $v^q\in\uIsomext(G^q,\Psi^q)(S)$ be the orientation induced by the inclusion of $T^q$ to $G^q$.
Let $u=c\cdot v^q\in \uIsomext(G,\Psi^q)(S)$ (by \ref{pairing}) and consider the homogeneous $G$--space $X^q=\gE(G,\Psi^q, u)$.
This is a left $G$ homogeneous space whose local stabilizers are maximal tori.

\smallskip
\noindent
{\bf Claim.} $X$ is an $S$-form of $X^q$.

Denote by  $\uIsom_G(X^q,X)$ the sheaf of $G$-space isomorphisms over $S$.
Since $\Phi(G_{S_i},H_i)$ is fppf-locally isomorphic to $\Psi^q_{S_i}$ over $S_i$,
we can find an fppf-cover ${S_{ij}}$ of $S_i$ such that there is $f\in X^q(S_{ij})$ with $f(T^q_{S{ij}})=H_{i,S_{ij}}$.
As the stabilizer $\uStab_{G_{S_{ij}}}(f)$ is $H_{i,S_{ij}}$, we have $X^q_{s_{ij}}\simeq G_{S_{ij}}/H_{i,S_{ij}}$ as $G_{S_{ij}}$-spaces.
Hence there is an isomorphism between $X^q$ and $X$ over $S_{ij}$ that preserves the $G$-structure.
This proves that $\uIsom_G(X^q,X)$ is nonempty.
The canonical isomorphism $\uIsom_G(X^q,X)\land^{\uAut_G(X^q)} X^q\to X$ then shows that $X$ is an $S$-form of $X^q$.
Our claim is established.

\smallskip

By identifying $W(\Psi^q)$ with $\uAut_G(X^q)$ as in Lemma \ref{lem_compatibilities} (1),
we denote by $\Psi$ the root datum $\uIsom_G(X^q,X)\land^{\uAut_G(X^q)} \Psi^q$.
Let $c'\in\uIsomext(\Psi^q,\Psi)(S)$ be the canonical orientation, and $v$ be the orientation
$u\cdot c'$ by pairing (\ref{pairing}).
We conclude that $X$ is isomorphic to the embedding functor $\gE(G,\Psi,v)$ by Lemma \ref{lem_compatibilities}(3).

\smallskip

\noindent (2)
It follows from Lemma \ref{lem_compatibilities}(2).
\end{proof}

\section{Weil restriction of twisted root data and isotypic decomposition of oriented embedding functors}
Let $h:S'\to S$ be a finite \'etale cover.
If $\cF'$ is an $S'$-functor we consider its Weil restriction $h_*(\cF')$
defined by $h_*(\cF')(X)= \cF'(X \times_S S')$ for every $S$-scheme
$X$ \cite[\S 7.6]{BLR}. 
In the following, for an $S$-scheme $X$,  we denote the base change of $h$ by $h_X$.

Let $Y'$ be an affine $S'$-scheme and consider the $S'$-functor $\Hom_{S'}(-, Y')$. The $S$-functor $h_*(\Hom_{S'}(-, Y'))$ is represented by 
an affine scheme  denoted by $\gR_{S'/S}(Y')$.

\begin{slemma} \label{lem_weil}
(1) Let  $G'$ be  a reductive (resp.\ semisimple, semisimple adjoint, 
semisimple simply connected) $S'$--group scheme.
Then $G=\gR_{S'/S}(G')$ is a reductive
(resp.\ semisimple, semisimple adjoint, 
semisimple simply connected)  $S$-group scheme.

\smallskip

\noindent (2)
Let  $T'$ be  an  $S'$--torus.
Then $\gR_{S'/S}(T')$ is an  $S$--torus $T$.
Furthermore we have a canonical isomorphism 
$ h_*(\cM')  \simlgr \cM $ where $\cM$ (resp.\ $\cM'$) stands for the character functor 
of $T$ (resp.\ $T'$). Similarly for the  cocharacter functors $\cN$ and $\cN'$, 
we have  canonical isomorphisms
$\cN  \simlgr h_*(\cN')$ and $\cN \simlgr \cM^\vee$.

\end{slemma}

\begin{proof} (1) We know  that $G$ is affine and smooth \cite[\S 7.6, prop.\ 5]{BLR}.
It remains to check that the  geometric fibers of $G$ are
reductive. Since Weil restriction commutes with base change, we
are reduced to the case $S=\Spec(k)$ where $k$ is
an algebraically  closed  field. Then $S'=\Spec(k) \sqcup \dots \sqcup \Spec(k)$ 
($d$ times) so that $G= G' \times_k \dots \times_k G'$.
Since a product of reductive algebraic groups is reductiven $G$ is reductive.
The other facts are clear since they can be checked fiberwise.

\smallskip

\noindent (2)  The $S$--group scheme $T=\gR_{S'/S}(T')$ is reductive by (1)
but also commutative so that $T$ is an $S$--torus.
We start with the cocharacter case.
For an $S$-scheme $X$, set $X'=S'\times_S X$. We have
\begin{align*}
\cN(X) &= \Hom_{X-gp}\bigl( \GG_{m,X}, \gR_{S'/S}(T') \times_S X \bigr)\\
\   &= \Hom_{X-gp}\bigl( \GG_{m,X}, R_{X '/X}(T'_{X'}) \bigr)  \\
\  &=  \Hom_{X'}( \GG_{m,X'}, T'_{X'}  ) \  (\mbox{by the definition of Weil restrition})\\
\  &=h_*(\cN')(X).
\end{align*}

For the character case, we use the norm $N_{S'/S}: \gR_{S'/S}(\GG_{m,S'}) \to \GG_{m,S}$
\cite[XVII, 6.3.13]{SGA4}. For an $S$-scheme $X$, we consider the map

\begin{align*}
h_*(\cM')(X)= \Hom_{X '-gp}(  T'_{X'} ,\GG_{m,X'} )
&\xrightarrow{h_*} \Hom_{X-gp}\bigl(  T_X , R_{X'/X}(\GG_{m,X '} ) \bigr)\\
& \to  \Hom_{X-gp}\bigl(  T_X , \GG_{m,X} ) \bigr)
=\cM(X),
\end{align*}
where the last map is the composition by the norm map.
We have then a morphism of $S$--functors $h_*(\cM') \to \cM$ 
and we claim that this is an isomorphism.
This is a map of fppf $S$--sheaves so up to localize for the \'etale topology
we can assume that $S'=S \sqcup \dots \sqcup S$ ($d$ times).
In this case we have $T'=T_1  \sqcup T_2 \dots \sqcup T_d$
where the $T_i$'s are $S$--tori.  In this case 
we have $R_{S'/S}(T')= T_1 \times_S \dots \times_S T_d$
the above map  can be explicited as follows  
 {\tiny
\[
\xymatrix{
\Hom_{X \times_S S'-gp}(  T'_{X'} ,\GG_{m,X \times_S S'} )
\ar[r]^{h_*} & \Hom_{X-gp}\bigl(  T_X , R_{X'/X}(\GG_{m,X \times_S S'} ) \bigr)
\ar[r] &  \Hom_{X-gp}\bigl(  T_X , \GG_{m,X} ) \bigr) \\
\prod_{i=1}^d \Hom_{X-gp}( T_{i,X} ,\GG_{m,X} ) \ar@{=}[u]
\ar[r]^{\prod \qquad} &\Hom_{X-gp}\bigl(  T_{1,X} \times_X T_{2,X} \dots  \times_X T_{d,X}, \GG_{m,X}^d  \bigr)
\ar[r]  \ar@{=}[u] &  \Hom_{X-gp}\bigl(  T_X , \GG_{m,X} ) \bigr)  \ar@{=}[u] .
}
\]
}
The bottom map applies an element $(\chi_i)_{i=1,...,d}$ 
on the character $\chi(t_1, \dots , t_d)= \chi_1(t_1) \dots \chi_d(t_d)$.
We conclude that the map $h_*(\cM') \to \cM$ is an isomorphism.
\end{proof}

Let $T'$ be an $S'$--torus and consider its Weil restriction  $T=\gR_{S'/S}(T')=h_*(T')$ 
which is an $S$-torus (Lemma \ref{lem_weil}.(2)).  

We consider  a root datum $\Psi'=(\cM', {\cM'}^\vee, \cR', {\cR'}^\vee)$
and would like to define its Weil restriction $h_*(\Psi')$
in such a way that it provides the product of twisted root data
in the case of the  split \'etale cover of rank $d$ (see 
Remark  \ref{rem_product}).

We cannot define it by an adjunction property because the
diagonal is not a morphism in the   
category of twisted root data.
One naive candidate  for the Weil restriction 
is $( h_*(\cM'), h_*({\cM'}^\vee), h_*(\cR'), h_*({\cR'}^\vee))$, 
but in the case of a split cover, it does not give the right answer
for roots and coroots.

\begin{sproposition} \label{prop_weil}
There exists a unique $S$--functor $\Psi' \mapsto h_*\Psi'$
from the category of twisted $S'$--root data  to the category of
 twisted $S$--root data  satisfying the 
following properties:

\smallskip
(i)  If  $\Psi'=(\cM', {\cM'}^\vee, \cR', {\cR'}^\vee)$, then 
$h_*\Psi'= ( h_*(\cM'), \cR,  h_*({\cM'}), \cR^\vee)$
for suitable $\cR$ and $\cR^\vee$;

\smallskip

(ii) For each $S$--scheme $X$, set $X'=X\times_S S'$. For each $X$--isomorphism
\break $u: X \times_S S' \to X \times_S S'$,  and  for
each  twisted root $X'$--datum  $\Psi'=(\cM', {\cM'}^\vee, \cR', {\cR'}^\vee)$,
 the automorphisms   $c_u: u^*\cM' \simlgr \cM'$,
$c^\vee_u: u^*\cM'^{\vee} \simlgr \cM'$
induce an isomorphism $t_u: h_{X,\ast}(u^*\Psi') \simlgr h_{X,\ast}\Psi'$.

\smallskip

(iii) For each $S$--scheme $X$ such that  there exists an $X$--isomorphism
${u: X^{ \sqcup d}} \simlgr X \times_S S'$,
and  for each  twisted root $X'$--datum  $\Psi'=(\cM', {\cM'}^\vee, \cR', {\cR'}^\vee)$,
then $h_*(u^*\Psi')$ is  the product construction (see Remark \ref{rem_product}).
\end{sproposition}

\begin{proof}
The question is local 
so that we can assume 
$S$ is affine and that $S'$ is finite \'etale of 
rank $d$ over $S$.  
\smallskip

\noindent{\it First case: $S'=S^{\sqcup d}$ }.
Let $T'$ be the $S'$--torus whose Cartier dual is $\cM'_{S'}$.
We have $T'=T_1 \sqcup T_2 \dots \sqcup T_d$ where each $T_i$ is an $S$-torus 
for $i=1,..,d$. We have  
$\Psi_1 \sqcup \dots \sqcup \Psi_d = \Psi'$ 
where each $\Psi_i$ is a twisted root datum for the $S$--torus
$T_i$. We define $h_*\Psi'=\Psi_1 \times \dots \times \Psi_d$ ( Remark \ref{rem_product}); 
is a twisted root datum with respect to  the $S$--torus
$T=T_1 \times_S T_2 \dots \times_S  T_d$.
It satisfies the first  requirement.
For the second requirement, 
an $X$--isomorphism $u: X^{\sqcup d} \simlgr X^{\sqcup d} $
is given by an element $\sigma \in S_d(X)$.
Up to localize on $X$ we can assume that $\sigma$ is constant so that the 
map  $u$ applies the $i$-factor $X$ on the $\sigma^{-1}(i)$--factor.
We have $u^*(T_X)=T_{X,\sigma(1)} \times_X T_{X,\sigma(2)} \dots \times_X  T_{X,\sigma(d)}$
so that $h_*(u^*\Psi')= \Psi_{\sigma(1)} \times \dots \times \Psi_{\sigma(d)}$.
The map $t_u: \Psi_{\sigma(1)} \times \dots \times \Psi_{\sigma(d)} \to
\Psi_1 \times \dots \times \Psi_d$ is the right map.
Property (ii) is established
and property (iii) follows of  (ii)
by using  that (iii) holds for the case of $S'$ itself. 
 The existence is then proven and unicity follows of property (iii).
 The fact that this construction is functorial is left to the reader.

\smallskip

\noindent{\it General  case.}
There exists a finite  \'etale cover $V \to S$ such that 
there exists a trivialization $u: V^{ \sqcup d} 
\simlgr S' \times_S V $. 
We put $V'=  S'\times_S V$.
The unicity for  the construction of $h_*\Psi$  follows then from the first 
case and of property (ii).

The first case defines the  twisted root datum
$\Psi^\sharp=h_{V,*}(\Psi'_{V'})$ over $V$.
If $h_*\Psi'$ exists, property (ii) provides
an isomorphism $t_u:  \Psi^\sharp \simlgr (h_*\Psi')_V$.
The unicity statement applied to $V\times_S V$
provides a descent data for $(h_*\Psi')_V$
for the \'etale cover $V \to S$.
Since twisted root data satisfy fpqc descent (Lemma \ref{lem_descent}),
$(h_*\Psi')_V$ descends to a twisted root data $\Psi$ on $S$ as desired.
By descent $\Psi$ satisfies all requirements.
\end{proof}

The twisted root data provided by Proposition \ref{prop_weil}
is called the Weil restriction of $\Psi'$ and is
denoted by $\gR_{S'/S}(\Psi')=h_*( \Psi')$. In the case where $S'= S^{\sqcup d}$, we see that the construction of Weil restriction of root data commutes with base change.
By descent argument, we see that the Weil restriction of root data commutes with base change in general.

For a twisted root datum $\Psi$ over $S$, denote the sheaf of outer automorphisms of $\Psi$ by $\uAutext(\Psi)$.
It is the quotient sheaf of $\uAut(\Psi)$ by the Weyl group $W(\Psi)$.
On the other hand, let $\uDyn(\Psi)$ be the $S$-Dynkin diagram
of $\Psi$ and $\uDynn(\Psi)$ its $S$-scheme of connected components
\cite[\S XXIV.5.2]{SGA3}.
The Weil restriction is functorial in $S$ and
satisfies the following expected properties (and more as transitivity
for a cover $S''/S'/S$). 
For an $S'$-scheme $X'$, we
remind the notation $_S\!X'$ for the $S$-scheme $X' \to X \to S$;
this functor is called sometimes Grothendieck's restriction and
is left adjoint to the base change functor \cite[\S I.1, 6.4 and 6.5]{DG}.

\begin{slemma} \label{lem_weyl} 
(1) Assume that   $S'=S'_1 \sqcup S'_2$ where $S'_1$ (resp.\  $S'_2$)
is a finite \'etale cover of $S$ with constant rank $d_1$ (resp.\ $d_2$),
and write  $\Psi'_i= \Psi'_{\mid S'_i}$ for $i=1,2$.
Then we have  decompositions
$$
\gR_{S'/S}(\Psi')= \gR_{S'_1/S}(\Psi'_1) \times \gR_{S'_2/S}(\Psi'_2),$$
$$ \uDyn\bigl( \gR_{S'/S}(\Psi')\bigr)= \uDyn\bigl(\gR_{S'_1/S}(\Psi'_1)\bigr)
\sqcup  \uDyn\bigl(\gR_{S'_2/S}(\Psi'_2)\bigr)$$ and $$
\uDynn\bigl( \gR_{S'/S}(\Psi')\bigr)= \uDynn\bigl(\gR_{S'_1/S}(\Psi'_1)\bigr) 
\sqcup  \uDynn\bigl(\gR_{S'_2/S}(\Psi'_2)\bigr). $$

 \smallskip
 
\noindent (2)
We have an  isomorphism between Weyl groups
$ W(h_*( \Psi'))  \simlgr \gR_{S'/S}( W(\Psi'))$.

 \smallskip
 
\noindent (3)  We have  isomorphisms
$ {_S\bigl(\uDyn(\Psi')\bigr)}  \, \simlgr \,   \uDyn(h_*( \Psi'))$
and \break
$ {_S\bigl( \uDynn(\Psi') \bigr)}  \, \simlgr \, \uDynn(h_*( \Psi'))$.

\end{slemma}

\smallskip

\begin{proof} We put $\Psi=\gR_{S'/S}(\Psi')$. 
\smallskip

\noindent (1) It is clear from the construction of the Weil restriction.

\smallskip

\noindent (2) 
Putting $S''=S'$ (and $\Psi''=\Psi' \times_{S'}S''$)
and  we decompose $S' \times_S S'' = S'' \sqcup S^+$ 
(where the first summand is the diagonal). The statement  (1) provides the following decomposition
of twisted $S''$--root data
$$
\Psi_{S''}= \Bigr(\gR_{S'/S}( \Psi')\Bigl)_{S''}
= \gR_{S'' \sqcup S^+ /S''}( \Psi'_{S'\times_S S''})
= \Psi'' \times \Psi^+ .
$$
It follows that $W(\Psi)_{S''}= W(\Psi'') \times_{S''} W(\Psi^+)$
and we consider the first projection $W(\Psi)_{S''} \to W(\Psi'')$.
The adjunction property
$$
\Hom_S\bigl(   W(\Psi),  \gR_{S'/S}(W(\Psi') \bigr)  \simlgr  
\Hom_{S'}\bigl(   W(\Psi)_{S'},  W(\Psi') \bigr) 
$$  
defines  a natural $S$--homomorphism of group schemes
$W(\Psi) \to \gR_{S'/S}( W(\Psi'))$.
By descent to the split case, this map is an $S$-isomorphism.

\smallskip

\noindent (3) 
The decomposition  $\uDyn(\Psi)_{S'}=\uDyn(\Psi') \sqcup \uDyn(\Psi^+)$
 provides an $S'$--map \break $f': \uDyn(\Psi') \to \uDyn(\Psi)_{S'}$. 
The adjunction property 
$$
\Hom_S\Bigl(   {_S\bigl(\uDyn(\Psi')\bigr)},  \uDyn(\Psi) \Bigr)  \simlgr  
\Hom_{S'}\bigl(  \uDyn(\Psi'),  \uDyn(\Psi)_{S'} \bigr) 
$$  
attaches to $f'$ an $S$--map $f: {_S\bigl(\uDyn(\Psi')\bigr)} \to   \uDyn(\Psi)$.
Once again, by descent to the split case, this map is an $S$-isomorphism.
The case of $\uDynn$ is similar.

\end{proof}

\subsection{Case of constant type}
Let $\Psi_0$ be a root data.
We focus on the case when the twisted root datum $\Psi'$ over $S$ is of  constant type $\Psi_0$.
We use the method of \cite[\S 4]{GNR} by considering
for any $S$-scheme $T$ the following groupoid $\cC(T)$. 
The  objects of  $\cC(T)$ are the 
 pairs $(T', \Upsilon')$
where $T'$ is finite \'etale cover of degree $d$ of $T$
and $\Upsilon'$ is a root data over $T'$
of constant type $\Psi_0$;
A morphism $(T'_1, \Upsilon'_1) \to (T'_2, \Upsilon_2')$
consists in  a pair $(u,v)$ where $u: T'_1 \simlgr T'_2$
and $v: \Upsilon'_1 \simlgr u^*\Upsilon'_2$.
Since all objects of $\cC(T)$ are locally isomorphic
for the \'etale topology to the split objects 
$(T^{\sqcup d}, \Psi_{0,T^{\sqcup d}})$, $\cC$ is an $S$--stack.

\begin{slemma} \label{lem_Theta} 
Set $\Theta_0=(S^{\sqcup d}, \Psi_{0,S^{\sqcup d}})$.

\noindent (1) $\uAut_S(\Psi_{0,S})^d \rtimes S_d \simlgr \uAut(\Theta_0)$
and we have a monomorphism $$\uAut(\Theta_0) \to \uAut_S(\Psi_{0,S} \times \dots \times \Psi_{0,S}).$$

\smallskip

\noindent (2) Assume furthermore than $\Psi_0$ is irreducible.
Then the map \break $\uAut(\Theta_0) \to \uAut_S(\Psi_{0,S} \times \dots \times \Psi_{0,S})$ is an isomorphism and we  have   isomorphisms
\begin{equation}\label{seq_31}
\uAut_S(\Psi_{0,S})^d  \rtimes S_d \simlgr 
\uAut_S( \Psi_{0,S} \times \dots \times \Psi_{0,S})
\end{equation}
and
\begin{equation}\label{seq_32}
\uAutext_S(\Psi_{0,S})^d  \rtimes S_d \simlgr 
\uAutext_S( \Psi_{0,S} \times \dots \times \Psi_{0,S}).
\end{equation}

\smallskip

\noindent (3) Assume that $\Psi_0$ is irreducible and let
$D_0$ be its Dynkin diagram. Then $\rDyn(\Psi_0^d)= D_0^{\sqcup d}$
and its set of connected components  $\rDyn_{\triangledown}(\Psi_0^d)$
is $\{1,\dots, d\}$. Furthermore we have a commutative diagram
\begin{equation}\label{seq33}
\xymatrix{\uAutext_S(\Psi_{0,S})^d  \rtimes S_d \ar[r]^\sim \ar[d] &
\uAutext_S( \Psi_{0,S} \times \dots \times \Psi_{0,S}) \ar[d]  \\
S_d   \ar[r]^\sim &   \uAut_\Dyn(\uDyn_\triangledown(\Psi^d_{0,S}))  .
}
\end{equation}
where the right vertical map  is that of \ref{seq_32}.

\end{slemma}

\begin{proof} 
We write $\Psi_0=(M_0, R_0, M_0^\vee, R_0^\vee)$.

\smallskip

\noindent (1) 
Since the symmetric group $S_d= \uAut_S(S^{\sqcup d})$, we have 
an $S$--homomorphism $\nu: \uAut(\Theta_0) \to S_d$
which is split by the permutation action. 
For each $S$-scheme $X$, we have
$\ker(\nu)(X)=\uAut_{S^{\sqcup d}}( \Psi_{0,S^{\sqcup d}})(X)=
\Aut_{X^{\sqcup d}}(\Psi_{0,X^{\sqcup d}})=\uAut_S(\Psi_{0,S})(X)^d$.
Thus $\uAut(\Theta_0)= \uAut_S(\Psi_{0,S}) ^d\rtimes S_d$
as desired. It provides a monomorphism
$ \uAut(\Theta_0) \to \uAut_S(\Psi_{0,S} \times \dots \times \Psi_{0,S})$.

\smallskip

\noindent (2)  
By definition an $S$-isomorphism $u: \Psi_{0,S} \times \dots \times \Psi_{0,S} \simlgr \Psi_{0,S} \times \dots \times \Psi_{0,S}$
is an isomorphism $u: M_{0,S}^d \simlgr M_{0,S}^d$
preserving roots and coroots.
There exists a partition $S= \sqcup_i U_i$ in clopen subsets
such that $u_{S_i}$ is constant over each $U_i$, that is
given for some $u_i : M_0 ^d\simlgr M_0^d$.
Since $\Psi_0$ is irreducible, it defines
a permutation $\sigma(u_i) \in S_d(U_i)$
whence an element $\sigma(u) \in S_d(S)$.
This assignment  defines an $S$-group homomorphism
${\gamma_0: \uAut_S(\Psi_{0,S} \times \dots \times \Psi_{0,S}) \to S_d}$ which is split by considering the 
permutation isomorphisms of $\Psi_0 \times \dots \times \Psi_0$. Next we have $\ker(\gamma)= 
\prod\limits_{i=1}^d \uAut_S(\Psi_{0,S})$, whence the first 
exact sequence and the isomorphism
$\uAut(\Theta_0) \simlgr \uAut_S(\Psi_{0,S} \times \dots \times \Psi_{0,S})$.

 The second sequence is obtained
by moding out by the Weyl group
$W(\Psi_{0,S})^d = W(\Psi_{0,S}^d)$.

\smallskip

\noindent (3) This is straightforward.

\end{proof}

 We come back now to the pair $\Theta=(S',\Psi')$ which is then  the twist of  $\Theta_0$ by the 
$\uAut(\Theta_0)$-torsor $E= \uIsom(\Theta_0,\Theta)$
which
locally trivial for the \'etale topology on $S$.
We consider the sheaf  $\uAut_S(\Psi_{0,S}^d)$-torsor
 $F= E \wedge^{\Aut(\Theta_0)} \uAut_S(\Psi_{0,S}^d)
= \uIsom\Bigl( \Psi_{0,S}^d, \, {}^E(\Psi_{0,S}^d))$. 
Twisting the above facts by $E$ leads to the next statement.

\begin{slemma}\label{lem_automorphism_Weil}
(1) We have an isomorphism
 $\gR_{S'/S}(\Psi') \simlgr 
\, ^F\! (\Psi_0 \times \dots \times \Psi_0)$ and
a monomorphism $\uAut(\Theta) \to  \uAut(\gR_{S'/S}(\Psi'))$.

\smallskip

\noindent (2) We have the following exact sequence of $S$--groups (in big \'etale site):
\begin{equation}\label{seq_1}
\begin{tikzcd}
1\to \gR_{S'/S}(\uAut(\Psi')) \to \uAut(\Theta) \xrightarrow{\gamma} \uAut_S(S') \to 1.
\end{tikzcd}
\end{equation}

\smallskip

\noindent (3) Assume that $\Psi_0$ is irreducible.
Then we  have an isomorphism 
 $\uAut(\Theta) \simlgr  \uAut(\gR_{S'/S}(\Psi'))$ and an exact sequence of $S$-groups (for the big \'etale site)

\begin{equation}\label{seq_2}
\begin{tikzcd}
1\to \gR_{S'/S}(\uAut(\Psi')) \to \uAut(\gR_{S'/S}(\Psi')) \xrightarrow{\gamma} \uAut_S(S') \to 1
\end{tikzcd}
\end{equation}

and

\begin{equation}\label{seq_3}
\begin{tikzcd}
1\to \gR_{S'/S}(\uAutext(\Psi')) \to \uAutext(\gR_{S'/S}(\Psi')) \to \uAut_S(S') \to 1.
\end{tikzcd}
\end{equation}

\end{slemma}

\begin{proof} We put $\Psi=\gR_{S'/S}(\Psi')$.
\smallskip

\noindent (1) In view of Proposition \ref{prop_weil},
$\Psi$ is an $S$--form of $
\gR_{S^{\sqcup d}/S}( \Psi_{0,S^{\sqcup d}})
= \Psi_{0,S}^{d}$.
Since $\Aut(\Theta_0)$ acts on $\gR_{S^{\sqcup d}/S}( \Psi_{0,S^{\sqcup d}})$,  the yoga of forms
yields that $\Psi \simlgr \, ^E\!\bigl(\gR_{S^{\sqcup d}/S}( \Psi_{0,S^{\sqcup d}})\bigr)$. In other words,
$\Psi$ is isomorphic to $^F\!(\Psi_{0,S}^d)$. 
Twisting the monomorphism $\uAut(\Theta_0)\to \uAut_S(\Psi_{0,S}^d)$ yields a monomorphism 
 $\uAut(\Theta) \to  \uAut(\Psi)$.

\smallskip

\noindent (2) 
Twisting  by $E$ the exact sequence
$$
1 \to \uAut_S(\Psi_{0,S})^d  \to \uAut(\Theta_0) \to S_d \to 1
$$
arising from Lemma \ref{lem_Theta}.(1)
gives rise to an exact sequence
\begin{equation} \label{seq43}
1 \to   {^E\!\bigl(}\uAut_S(\Psi_{0,S})^d\bigr)
\to \uAut(\Theta) \xrightarrow{\nu} \uAut_S(S') \to 1.
\end{equation}
It remains to identify  ${^E\!\bigl(}\uAut_S(\Psi_{0,S})^d\bigr)$
with $\gR_{S'/S}(\uAut(\Psi'))$. 
We have a  natural monomorphism $\gR_{S'/S}(\uAut(\Psi')) \to
\ker(\nu)$ which becomes after a suitable \'etale  localization on $S$ the morphism $\uAut(\Psi_0)^d \to \uAut(\Theta_0)$. It follows that the monomorphism
$\gR_{S'/S}(\uAut(\Psi')) \to \uAut(\Theta)$ 
provides an identification $\gR_{S'/S}(\uAut(\Psi'))
\simlgr {^E\!\bigl(\uAut_S(\Psi_{0,S})^d\bigr)}$.

\smallskip

\noindent (3) We assume that $\Psi_0$ is irreducible. Lemma \ref{lem_Theta}.(3)
 provides  an isomorphism
$\uAut(\Theta_0) \simlgr  \uAut_S(\Psi_{0,S} \times \dots \times \Psi_{0,S})$ whence an isomorphism $\uAut(\Theta) \simlgr  \uAut(\Psi)$ by (1). The sequence \eqref{seq_1} gives then
the sequence \eqref{seq_2}.  
By moding out by the Weyl $S$--group
$R_{S'/S}(W(\Psi')) \simlgr W(\Psi)$ (Lemma \ref{lem_weyl}),
we obtain the sequence \eqref{seq_3}.

\end{proof}

\begin{sremark} \label{rem_automorphism_Weil}
{\rm 
 Note that the map $h_\ast:\gR_{S'/S}(\uAut(\Psi'))\to
 \uAut(\Theta) \to  \uAut_S(\gR_{S'/S}(\Psi'))$ defined above is 
a monomorphism without the assumption on 
irreducibility on $\Psi_0$.
However we need the assumption that $\Psi_0$ is irreducible 
to conclude that $\uAut(\Theta) \simeq\uAut_S(\gR_{S'/S}(\Psi'))$ and hence to define a  morphism from $ \uAut(\gR_{S'/S}(\Psi'))$
  to $\uAut_S(S')$.
}
\end{sremark}
\smallskip

We come now to the isotypic decomposition
which is analogous to that for simply connected (resp.\ adjoint) semisimple $S$--group scheme \cite[\S XXIV.5]{SGA3}

\begin{sproposition} \label{prop_isotypic}
Assume that the root datum $\Psi_0$ is semisimple
simply connected (resp.\ adjoint).
Let $\Psi_0= \prod_{\bt} \Psi^\sharp_{0,t}$
be the isotypic decomposition \cite[XXI.7.1.6]{SGA3}
and choose an isomorphism $\Psi^\sharp_{0,\bt} \simlgr \Psi_{0,\bt}^{d_{\bt}}$  where $\Psi_{0,\bt}$ is irreducible and $d_{\bt}$ is a positive integer. 
Let $\Psi$ be a twisted root datum of constant type
$\Psi_0$; let $\gD=\uDyn(\Psi)$  be the 
$S$--Dynkin diagram (which is a finite $S$-\'etale cover of constant degree) of $\Psi$
and consider its isotypic decomposition $\gD= \amalg_\bt \, \gD_\bt$ \cite[XXIV.5.2]{SGA3}. 

\smallskip

\noindent (1)
We have a canonical isotypic decomposition
 $$
\Psi= \prod_{\bt} \Psi_{\bt}
$$ 
where each $\Psi_{\bt}$ is of type $\Psi^\sharp_{0,\bt}$.
Furthermore it induces identifications $\uDyn(\Psi_{\bt})= \gD_\bt$ 
for each $\bt$.

\smallskip

\noindent (2)  For each $\bt$, let $\gD_{\bt, \triangledown}$ be the $S$-scheme 
of connected components of $\gD_\bt$ (denoted by $(\gD_\bt)_0$
in \cite[XXIV.5.2]{SGA3}).
Then there exists an isomorphism 
 $\Psi_{\bt} \simlgr  \gR_{\gD_{\bt, \triangledown}/S}  (\Psi'_{\bt})$
 where  $\Psi'_{\bt}$ is a twisted root data over $\gD_{\bt, \triangledown}$  of 
 constant type $\Psi_{0,\bt}$. Furthermore  $\Psi'_{\bt}$
is unique up to isomorphism.

\end{sproposition}

\begin{proof}
(1) In view of   \cite[\S XXIV.5.4]{SGA3},
we have a decomposition $\uAut(\Psi_{0,S})
= \prod_{\bt} \uAut(\Psi^\sharp_{0,\bt, S})$.
Since $\Psi$ is an $S$-form of $\Psi_{0,S}$,
the yoga of forms provides a canonical decomposition 
$\Psi= \prod_{\bt} \Psi_{\bt}$ where 
each $\Psi_{\bt}$ is an $S$-form of $\Psi^\sharp_{0,\bt, S}$.

Let $D_0$ be the Dynkin diagram of $\Psi_0$,
it decomposes as $D_0= \amalg_{\bt} \, D^\sharp_{0,\bt}$.
Since we have a commutative diagram 
\[
\xymatrix{
\uAutext_S(\Psi_{0,S})  
\ar[r]^{\sim} \ar[d]^{\wr}   & \prod_{\bt} \uAutext(\Psi^\sharp_{0,\bt, S}) 
\ar[d]^{\wr} \\
\uAut_{\Dyn}(D_{0,S})  
\ar[r]^{\sim}  & \prod_{\bt} \uAut_{\Dyn}(D^\sharp_{0,\bt, S}),
}
\]
we obtain that $\uDyn(\Psi)= \amalg_{\bt} \, \uDyn(\Psi_{\bt})$ and
$\gD_\bt= \uDyn(\Psi_{\bt})$.

\smallskip

\noindent (2) We fix a type $\bt$.
We put $\Theta_{0,\bt}= 
\bigl( S^{\sqcup d_\bt}, \Psi_{0, \bt,S^{\sqcup d_\bt }}\bigr)$ and we have
$\uAut\bigl( \Theta_{0,\bt}\bigr) \simlgr \uAut_S(\Psi_{0,\bt,S}^{d_\bt})$ in view of Lemma \ref{lem_Theta}.(2). Since $\Psi_\bt$ is an $S$-form
of $\Psi_{0,\bt,S}^{d_\bt}$, it defines an $S$-form
$\Theta_\bt$ of $\Theta_{0,\bt}$, that is, a 
pair $(S'_\bt, \Psi'_\bt)$
where $S'_\bt \to S$ is a finite \'etale cover of 
$S$ of degree $d_\bt$ and $\Psi'_\bt$ is
 a twisted $S$--root data of constant type
$\Psi_{0,\bt}$. The point is that $\gR_{S'_\bt/S}(\Psi'_\bt)$ is isomorphic to $\Psi_\bt$  in view of
Lemma \ref{lem_automorphism_Weil}.(1).
The unicity of $(S'_\bt, \Psi'_\bt)$ follows of
the fact that there is an equivalence of categories
between $S$-forms of $\Theta_{0,\bt}$
and $S$-forms of $\Psi_{0,\bt,S}^{d_\bt}$
in view of the isomorphism $\uAut\bigl( \Theta_{0,\bt}\bigr) \simlgr 
\uAut(\Psi_{0,\bt,S}^{d_\bt})$. 
It remains to show that $S'_\bt$ is $S$--isomorphic to $\gD_{\bt, \triangledown}$.
Using the compatibility \eqref{seq33}, we observe that  the 
 isomorphism $\Psi^\sharp_{0,\bt,S} \simlgr \Psi^{d_\bt}_{0, \bt,S}$
induces a commutative diagram
\[
\xymatrix{
\uAut_S(\Psi^\sharp_{0,\bt,S})  
\ar[r]^{\sim} \ar[d]   & \uAut_S(\Psi^{d_\bt}_{0,\bt, S}) 
\ar[d] \\
\uAut_\Dyn(D^\sharp_{0,\bt,\triangledown,S} ) \ar[r]^{\sim}  & S_d
}
\]
where $D^\sharp_{0,\bt,\triangledown}$ stands for the set of connected components 
of the Dynkin diagram $D^\sharp_{0,\bt}$.
By twisting everything by the $\uAut(\Psi^\sharp_{0,S} )$--torsor
$\uIsom\bigl(\Psi^\sharp_{0,\bt,S}, \Psi_{\bt} \bigr)$, 
we obtain that $\gD_{\bt, \triangledown}$ is $S$--isomorphic to $S'_{\bt}$.
\end{proof}

\smallskip

\subsection{Admissibility and Weil restriction}

\begin{sproposition}\label{prop_admissible_Weil} 
Let $\Psi_0$ be a root datum.
Let $\Psi'$ be a  twisted root datum defined over $S'$
of constant type $\Psi_0$ and let $G'$ be a reductive
$S'$-group scheme of constant type $\Psi_0$.

\smallskip
\noindent 
\begin{enumerate}
\item The reductive $S$--group scheme $\gR_{S'/S}(G')$
has constant type $\Psi_0^d$ and so has  $\gR_{S'/S}(\Psi')$.
 
 \smallskip
 
\item There is a canonical  monomorphism
$$
\gh:\gR_{S'/S}\Bigl(\uIsomext(G',\Psi')\Bigr) \to 
\uIsomext\Bigl(\gR_{S'/S}(G'),\gR_{S'/S}(\Psi') \Bigr).
$$
\end{enumerate}
\end{sproposition}

\begin{proof}
Set $\Psi=\gR_{S'/S}(\Psi')$ and $G=\gR_{S'/S}(G')$.

\noindent 

(1) Let $\ol{s}$ be a geometric point of $S$.
As $h:S'\to S$ is a finite \'etale cover, $S'\times_S \ol{s}\simeq \ol{s}_1\sqcup\dots\sqcup\ol{s}_d$ with $\ol{s}_i\simeq \ol{s}.$
Set $\ol{s}'=S'\times_S \ol{s}$.
Then $G'_{\ol{s}'}\simeq G_1\sqcup\cdots\sqcup G_d,$ where each $G_i$ is a reductive group over $\ol{s}_i$ of type
$\Psi_0$. It follows that $G'_{\ol{s}'}$ has
type $\Psi_0^d$.
Similarly $\Psi'_{\ol{s}'}\simeq \Psi_0\sqcup\dots\sqcup\Psi_0$.

\noindent (2) 
Recall that for an $S$-scheme $X$, we denote $S'\times_S X$ by $X'$ and let $h_X:X'\to X$ be the base change of $h:S'\to S$. 

Let $v\in\gR_{S'/S}(\uIsomext(G',\Psi'))(X)=\uIsomext(G',\Psi')(X')$. 
We aim to define an orientation $u\in\uIsomext(G,\Psi)(X)$ from $v$ in a canonical way.

Since $h_X: X'\to X$ is a finite \'etale cover, we can choose an \'etale cover $\{U_i\to X\}$ satisfying the following properties \label{split cover}:
\begin{enumerate}
\item[(C1)] The cover splits $X'$, i.e. $U_i\times_X X'\simeq \underset{l}{\amalg}\  V_{i,l}$, where $V_{i,l}\simeq U_i$ as $X$-schemes for each $V_{i,l}$.
\item[(C2')] The group $G'_{U'_i}$ has a maximal torus, where $U'_i=U_i\times_X X'$.
\end{enumerate}

Choose a maximal torus ${T^{\diamond}_i}^\prime$  of $G'_{U'_i}$ for each $i$ and set $\Psi^{\diamond\prime}_i=\Phi(G'_{U'_i},{T^{\diamond}_i}^\prime)$. 
Let $T^\diamond_i=\gR_{U'_i/U_i}(T^{\diamond\prime}_i)$ and $\Psi^\diamond_i=\gR_{U'_i/U_i}( \Psi^{\diamond\prime}_i)$. Then by the construction of $\Psi^\diamond_i$ we have $\Psi^\diamond_i=\Phi(G_{U_i},T^\diamond_i)$ as twisted root data over $U_i$.
The morphism $h_{U_i,\ast}: \gR_{U'_i/U_i}(\uIsom(\Psi^{\diamond\prime}_i,\Psi'_{U'_i}))\to \uIsom(\Psi^\diamond_{i},\Psi_{U_i})$ gives an isomorphism 
$$\uAut(\Psi_{U_i})\land^{\gR_{U'_i/U_i}(\uAut(\Psi'_{U'_i}))}\gR_{U'_i/U_i}(\uIsom(\Psi^{\diamond\prime}_i,\Psi'_{U'_i}))\to\uIsom(\Psi^\diamond_i,\Psi_{U_i}),$$ 
which in turn induces an isomorphism by Lemma \ref{lem_automorphism_Weil}.
$$\imath:\uAutext(\Psi_{U_i})\land^{\gR_{U'_i/U_i}(\uAutext(\Psi'_{U'_i}))}\gR_{U'_i/U_i}(\uIsomext(\Psi^{\diamond\prime}_i,\Psi'_{U'_i}))\to\uIsomext(\Psi^\diamond_i,\Psi_{U_i}).$$
As $\gR_{U'_i/U_i}(\uAutext(\Psi'_{U'_i}))$ is a subgroup of  $\uAutext(\Psi_{U_i})$ by Lemma \ref{lem_automorphism_Weil}, we have a canonical monomorphism from 
$\gR_{U'_i/U_i}(\uIsomext(\Psi^{\diamond\prime}_i,\Psi'_{U'_i}))$ to 
$\uAutext(\Psi_{U_i})\land^{\gR_{U'_i/U_i}(\uAutext(\Psi'_{U'_i}))}\gR_{U'_i/U_i}(\uIsomext(\Psi^{\diamond\prime}_i,\Psi'_{U'_i})$.
Compose this map with $\imath$, we get a monomorphism
 \[\gh_i:\gR_{U'_i/U_i}(\uIsomext(\Psi^{\diamond\prime}_i,\Psi'_{U'_i}))\to \uIsomext(\Psi^\diamond_i,\Psi_{U_i}).\]

\noindent If we choose another maximal torus $T^{\sharp\prime}_i$ of $G'_{U'_i}$ and let $\Psi^{\sharp\prime}_i=\Phi(G'_{U'_i},T^{\sharp\prime}_i)$, then there is a canonical isomorphism 
from $\uIsomext(\Psi^{\diamond\prime}_i,\Psi'_{U'_i})\land^{N_{G'}(T^{\diamond\prime}_i)^{op}}\uTransp(T^{\diamond\prime}_{i},T^{\sharp\prime}_{i})$ to 
$\uIsomext(\Psi^{\sharp\prime}_i,\Psi_{U'_i})$. This allows us to identify $\uIsomext(\Psi^{\diamond\prime}_i,\Psi'_{U'_i})$ with $\uIsomext(\Psi^{\sharp\prime}_i,\Psi_{U'_i})$
in a canonical way and to define $\uIsomext(G',\Psi')$ (see \S 3.2). 

As $\gR_{U'_i/U_i}(\uTransp(T^{\diamond\prime}_{i},T^{\sharp\prime}_{i}))=\uTransp(T^\diamond_i,T^\sharp_i)$,  from $\gh_i$ we get
the monomorphism $\gh:\gR_{S'/S}(\uIsomext(G',\Psi'))\to \uIsomext(G,\Psi)$ by descent.

\end{proof}

We get from the first statement
the following corollary immediately.

\begin{scorollary}\label{coro_admissible_Weil}
Keep the notation as in Proposition \ref{prop_admissible_Weil}.
If $\Psi'$ is admissible for $G'$, then $\gR_{S'/S}(\Psi')$ is admissible for $\gR_{S'/S}(G')$.
\end{scorollary}
\smallskip

\subsection{Isotypic decomposition of embedding functors}

We relate now orientated embedding functors with isotypic decomposition
of semisimple simply connected (resp.\ adjoint)
$S$--group schemes \cite[\S XXIV.5]{SGA3}.
We assume that the root datum $\Psi_0$ is semisimple
simply connected (resp.\ adjoint). Let $\Psi$ be a twisted root datum of constant type
$\Psi_0$ and consider the  isotypic decomposition 
\begin{equation} \label{dec0}
\Psi = \prod_{\bt} \Psi_\bt \simlgr \prod_{\bt} \gR_{\uDynn(\Psi_\bt)/S}(\Psi'_\bt)
\end{equation}
as in Proposition  \ref{prop_isotypic}. In particular, if  
 $T$ is $S$--torus defined by $\Psi$, we have 
 a decomposition $T \simlgr \prod_{\bt} T_\bt$
 where each $T_\bt$ is the $S$-torus defined by $\Psi_\bt$.
 Let $T'_\bt$ be the $\uDynn(\Psi_\bt)$-torus defined by $\Psi'_\bt$.

Let $G$ be a semisimple simply connected
(resp.\ adjoint) $S$--group scheme
of constant type $\Psi_0$. 
We consider its  isotypic decomposition \cite[XXIV.5.10]{SGA3}
\begin{equation} \label{dec1}
G\simlgr \prod_{\bt} G_\bt = \prod_{\bt} \gR_{\uDynn(G)_\bt/S}(G'_\bt)  
\end{equation}
where $\uDynn(G)_\bt$ stands for the $S$--sheaf of connected
components of the Dynkin diagram $\uDyn(G)_\bt$. 
We assume that $\Psi$ is admissible for 
$G$ and we are given an orientation  $v \in \uIsomext(G,\Psi)(S)$.
In view of  \cite[XXIV.3.6]{SGA3}, the morphism $\uIsomext(G,\Psi) \to  
\uIsom_{\rDyn}\bigl(\uDyn(G),\uDyn(\Psi)\bigr)$ is an isomorphism and similarly for the
isotypic components. We can then discuss everything in term
of isomorphisms of Dynkin diagrams.
We have then a commutative diagram
\begin{equation} \label{square0}
\xymatrix{
\uIsomext(G,\Psi)     \ar[r]\ar[d]^{\wr}   & \prod_\bt \uIsomext(G_\bt,\Psi_\bt)
   \ar[d]^{\wr} \\  
\uIsom_{\rDyn}(\uDyn(G),\uDyn(\Psi)) 
\ar[r] ^{\sim}    & \prod_\bt \uIsom_{\rDyn}(\uDyn(G_\bt),\uDyn(\Psi_\bt)) 
}
\end{equation}
so that the top horizontal map is an $S$-isomorphism.
It follows that $v$ defines  orientations $v_\bt \in \uIsomext(G_\bt,\Psi_\bt)(S)$.
We claim that we  have a natural $S$--isomorphism 
\begin{equation} \label{eq62}
 \prod_{\bt} \gE(G_\bt,\Psi_\bt, v_\bt) \simlgr \gE(G,\Psi,v)
\end{equation}
defined  by the following assignment.
For an $S$-scheme $X$ and  $X$-tori embeddings $f_\bt: T_\bt \hookrightarrow G_{\bt, X}$
such that $f_\bt$ induces an isomorphism  $\Psi_{\bt,X} \simlgr \Phi(G_{\bt, X}, f_\bt(T_\bt))$ 
 which  induces  the orientation $v_\bt$, we associate the product $\prod_\bt f_\bt: T_X =
\prod_\bt T_{\bt, X}  \xhookrightarrow{\prod_\bt f_\bt}   G_{X}$.
This $S$-map is an isomorphism and is equivariant with respect to 
the decomposition  $G=\prod_\bt G_\bt$.

We focus now on a single isotypic component  of type $\bt$.
Denoting 
\[v_{\rDyn,\bt} \in  \uIsom_\Dyn(\uDyn(G_\bt),\uDyn(\Psi_\bt))(S)
\simeq\uIsomext(G_\bt,\Psi_\bt)(S)\]
the isomorphism given by the orientation $v_\bt$, we obtain also an  isomorphism $v_{\rDyn,\bt, \triangledown}: \uDynn(G)_\bt \simlgr \uDynn(\Psi)_\bt$ by taking the connected components.
Next in view of \cite[XXIV.5.13]{SGA3}, we have commutative diagrams
\begin{equation} \label{square1}
\xymatrix{
\uDyn(G'_\bt)    \ar@{=}[r]\ar[d]^{\pi}  & \uDyn(G_\bt) \ar[d] \\  
 \uDynn(G_\bt) \ar[r]& S,
}
\end{equation}
and similarly 
\begin{equation} \label{square2}
\xymatrix{
\uDyn(\Psi'_\bt)   \ar@{=}[r]\ar[d]   & \uDyn(\Psi_\bt) \ar[d]\\  
 \uDynn(\Psi_\bt) \ar[r]& S,
}
\end{equation}
The map $v_{\rDyn,\bt}$ connects the two preceding diagrams
and provides then a cartesian square

\begin{equation} \label{square3}
\xymatrix{
\qquad v^\flat_{\rDyn,\bt} : &\uDyn(G'_\bt)   \ar[r]^{\sim} \ar[d]^{\pi}   & \uDyn(\Psi'_\bt)  \ar[d]\\  
 \qquad v_{\rDyn,\bt, \triangledown}: &\uDynn(G_\bt) \ar[r]^{ \sim} & \uDynn(\Psi_\bt) .
}
\end{equation}
In particular $v^\flat_{\rDyn,\bt}$ and $v_{\rDyn,\bt, \triangledown}$
defines an isomorphism 

\begin{equation} \label{square4}
v_{\rDyn,\bt}^\dagger: {_S\bigl(\uDyn(G'_\bt)\bigr)}
\, \simlgr  \,  {_S\bigl(\uDyn(\Psi'_\bt) \bigr)}
\end{equation}
which is nothing but $v_{\rDyn,\bt}$ when taking the identification
 ${_S\bigl(\uDyn(G'_\bt)\bigr)}= \uDyn(G_\bt)$
 (resp.\ ${_S\bigl(\uDyn(\Psi'_\bt) \bigr)}=
\uDyn(\Psi_\bt)$ provided  by Lemma \ref{lem_weyl}.(3)
applied to the finite \'etale cover 
$\uDynn(G_\bt)  \to S$ (resp.\ $\uDynn(\Psi_\bt)  \to S$).

By identifying $\uDynn(G_\bt)$ with $\uDynn(\Psi_\bt)$ via $v_{\rDyn,\bt, \triangledown}$, we can regard $G'_\bt$ and $\uDyn(G'_\bt)$ as schemes over  $\uDynn(\Psi_\bt)$.  
We have the following commutative diagrams:
\begin{equation} \label{square5'}
\xymatrix{
 G'_\bt \quad \ar[r]_{=}^{\id} \ar[d]^{\pi_G}   &  \quad G'_\bt   \ar[d]^{v_{\rDyn,\bt,\triangledown}\circ\pi_G} \\
\uDynn(G_\bt) \quad \ar[r]_\sim^{v_{\rDyn,\bt, \triangledown}}  &  \quad \uDynn(\Psi_\bt) . \\
 }
\end{equation}

and

\begin{equation} \label{square5}
\xymatrix{
 \uDyn(G_\bt)\ar@{=}[r] \ar@/^19pt/[rrrr]^{v_{\rDyn,\bt}} \ar[d] & \uDyn(G'_\bt) \quad  \ar[r]_{=}^{\id} \ar[d]^{\pi}    &  \quad \uDyn(G'_\bt)  \quad \ar[d]^{v_{\rDyn,\bt,\triangledown}\circ\pi} 
\ar[r]_{\sim}^{v^\flat_{\rDyn,\bt}} & \uDyn(\Psi'_\bt)  \ar[d] \ar@{=}[r]\ar[d] & \uDyn(\Psi_\bt)   \ar[d] \\  
S  \ar@{=} @/_21pt/[rrrr]& \uDynn(G_\bt) \quad 
\ar[r]_\sim^{v_{\rDyn,\bt, \triangledown}} \ar[l]& \quad \uDynn(\Psi_\bt) \quad  \ar[r]_{=}^{\id} &  \quad \uDynn(\Psi_\bt) \ar[r] & S. \\
 &&&&}
\end{equation}
The diagram (\ref{square5'}) gives an isomorphism $\imath$ from $G_\bt=\gR_{\uDynn(G_\bt)/S}(G'_\bt)$ to $\gR_{\uDynn(\Psi_\bt)/S}(G'_\bt)$.
This in turns gives an  $S$-isomorphism $\imath_{\rDyn}$ from $\uDyn(G_\bt)$ to 
$\uDyn(\gR_{\uDynn(\Psi_\bt)/S}(G'_\bt))$.
Note that $\uDyn(\gR_{\uDynn(\Psi_\bt)/S}(G'_\bt))={_S\bigl(\uDyn (G'_\bt)\bigr)}$ 
and $\uDyn\bigl((\gR_{\uDynn(G_\bt)/S}(G'_\bt)\bigr)={_S\bigl(\uDyn (G'_\bt)\bigr)}$ by Lemma \ref{lem_weyl}.(3) applied to 
the finite \'etale covers $\uDynn(\Psi_\bt) \to S$ and $\uDynn(G_\bt)\to S$ respectively.
Hence $\imath_{\rDyn}$ is nothing but the identity map on $_S\uDyn(G'_\bt)$.

The isomorphism $v^\flat_{\rDyn,\bt}\in \uIsom_\Dyn(\uDyn(G'_\bt),\uDyn(\Psi'_\bt))(\uDynn(\Psi_\bt))$ gives an orientation $v^\flat_\bt\in\uIsomext_{\uDynn(\Psi_\bt)}(G'_\bt,\Psi'_\bt)(\uDynn(\Psi_\bt))$. Note that  $v^\flat_{\rDyn,\bt}$ is also a section of $\uIsom_\Dyn\bigl(_S\uDyn (G'_\bt),\ _S\uDyn(\Psi'_\bt)\bigr)(S)$.
We claim that there is an $S$-isomorphism

\begin{equation} \label{dec2}
\gR_{\uDynn(\Psi_\bt)/S}\Bigl(  \gE(G'_\bt,\Psi'_\bt, v^\flat_{\bt})  \Bigr) 
\simlgr \gE(G_\bt,\Psi_\bt, v_\bt) 
\end{equation}
defined as follows. 
If  $X$ is an $S$-scheme,  we 
have
$$
\gR_{\uDynn(\Psi_\bt)/S}\Bigl(  \gE(G'_\bt,\Psi'_\bt, v^\flat_{\bt})  \Bigr) 
(X)= \gE(G'_\bt,\Psi'_\bt, v^\flat_{\bt})(X \times_S \uDynn(\Psi_\bt));
$$
an element of this set is an $X \times_S \uDynn(\Psi_\bt)$-monomorphism 
$$
f': T'_{\bt, X \times_S \uDynn(\Psi_\bt)} \hookrightarrow 
 G'_{\bt, X \times_S \uDynn(\Psi_\bt)}
 $$
such that $f'$ induces an isomorphism
 $ \Phi\bigr(G'_{\bt, X \times_S \uDynn(\Psi_\bt)},  f'(T'_{\bt, X \times_S \uDynn(\Psi_\bt)}  ) \bigr) \simlgr (\Psi'_\bt)_{ X \times_S \uDynn(\Psi_\bt)}$ 
 which induces the orientation $v^\flat_{\bt}$. Taking the Weil restriction for $\uDynn(\Psi_\bt)/S$
provides a monomorphism
\begin{equation}\label{eq63}
f^{\prime,Weil} : T_{\bt,X}=\bigl(\gR_{\uDynn(\Psi_\bt) /S}  (T'_{\bt} )\bigr)_X 
\quad \hookrightarrow \quad  
 \bigl(\gR_{\uDynn(\Psi_\bt)/S}(G'_\bt)  \bigr)_X
\end{equation}
such that $f^{\prime,Weil}$ induces an isomorphism
$$ \Phi\Bigl(\bigl(\gR_{\uDynn(\Psi_\bt)/S}(G'_\bt)\bigr)_X ,  f^{\prime, Weil}(T_{\bt,X}) \Bigr)  \simlgr \bigl(\gR_{\uDynn(\Psi_\bt)/S}(\Psi'_\bt)\bigr)_X=\Psi_{\bt,X}$$
which induces the orientation $v^\diamond$
  arising from 
\[
\xymatrix{
{v^\flat_{\Dyn,\bt}} \in &\uIsom_\Dyn\Bigl( {_S\bigl(\uDyn(G'_\bt)\bigr)} ,  {_S\bigl(\uDyn(\Psi'_\bt)\bigr)} \Bigl)(S) \ar@{=}[d] \\
&\uIsom_\Dyn\Bigl( \uDyn\bigl( \gR_{\uDynn(\Psi_\bt)/S}(G'_\bt) \bigr) , 
\uDyn(\Psi_{\bt} )\Bigl)(S) \\
v^\diamond_\bt \in &\uIsomext\Bigl( \gR_{\uDynn(\Psi_\bt)/S}(G'_\bt),\Psi_{\bt} \Bigr)(S) 
\ar[u]^\wr
}
\]
where we use in the first equality the identifications   $
{_S\bigl(\uDyn(G'_\bt)\bigr)} \simlgr  \uDyn\bigl( \gR_{\uDynn(\Psi_\bt)/S}(G'_\bt) \bigr)$ and 
$
{_S\bigl(\uDyn(\Psi'_\bt)\bigr)} \simlgr  \uDyn\bigl( \gR_{\uDynn(\Psi_\bt)/S}(\Psi'_\bt) \bigr)$
 from Lemma \ref{lem_weyl}.(3) applied to the finite \'etale cover 
 $\uDynn(\Psi_\bt)\to S$.

 Composing $f^{',Weil}$ by the isomorphism  $\imath^{-1}:\gR_{\uDynn(\Psi_\bt)/S}(G'_\bt)\simlgr G_\bt$ coming from the diagram \eqref{square5'},  
 we get a monomorphism
\begin{equation}\label{eq64}
f :   T_{\bt,X}  \hookrightarrow  G_{\bt,X}
\end{equation}
such that $f$ induces an isomorphism
$\Phi( G_{\bt,X}, f (T_{\bt,X})) \simlgr  \Psi_{\bt,X}$
which induces the orientation arising from the composition
$$
\uIsom_{\rDyn}\Bigl( \uDyn(G_\bt), \uDyn\bigl(\gR_{\uDynn(\Psi_\bt)/S}( G')\bigr)\Bigr) \times
\uIsom_{\rDyn}\Bigl(\uDyn\bigl(\gR_{\uDynn(\Psi_\bt)/S}( G')\bigr),  \uDyn(\Psi_\bt ) \Bigr) 
\to $$
$$
\uIsom_{\rDyn}\Bigl( \uDyn(G_\bt), \uDyn(\Psi_t) \Bigr) $$
to $\imath_{\rDyn}\circ {(v^\flat_{\rDyn,\bt})}  =
v_{\rDyn,\bt}$ in view of the factorization
\eqref{square5}.
The map \eqref{eq63} is then  well-defined and is equivariant
with respect to the isotypic decomposition \eqref{dec1}.
To check that it is an isomorphism is routine. 
Indeed by descent, we are reduced to the  direct product case which 
is clear. Combining  \eqref{dec2} and  \eqref{eq63} 
provides an isotypic decomposition 

\begin{equation} \label{dec3}
 \gE(G,\Psi,v) = \prod_{\bt} \, \gE(G_\bt,\Psi_\bt, v_\bt) \simlgr 
 \prod_{\bt}  \, \gR_{\uDynn(\Psi_\bt)/S}\Bigl(  \gE(G'_\bt,\Psi'_\bt, v^\flat_\bt)  \Bigr) 
\end{equation} 
 which is equivariant for the isotypic decomposition   \eqref{dec1}.

\begin{sremark}
 {\rm 
 Keep the notation as above.
Let $G=\gR_{S'/S}(G')$ and $\Psi=\gR_{S'/S}(\Psi')$, where $G'$ and $\Psi'$ are semisimple simply connected.
In general an orientation $v$ between $G$ and $\Psi$ over $S$ does not guarantee the existence of an orientation between $G'$ and $\Psi'$ over $S'$,
even when $G'$ is isotypic. To see this, assume both $G'$ and $\Psi'$ are isotypic of type $\bt$. 
Write $G'=\gR_{\uDynn(G')/S'}(G'_\bt)$ and $\Psi'=\gR_{\uDynn(\Psi')/S'}(\Psi'_\bt)$. Then $G=\gR_{\uDynn(G')/S}(G'_\bt)$ and $\Psi=\gR_{\uDynn(\Psi')/S}(\Psi'_\bt)$.
The orientation $v\in\uIsom_\Dyn(\Dyn(G),\Dyn(\Psi))(S)$ gives $v^\flat_{\rDyn,\bt}$ and $v_{\rDyn,\bt,\triangledown}$ in (\ref{square3}).
However, as $v_{\rDyn,\bt,\triangledown}$ is a morphism over $S$ which is not necessarily an $S'$-morphism, we do not have identification between 
$G'=\gR_{\uDynn(G')/S'}(G'_\bt)$ and $\gR_{\uDynn(\Psi')/S'}(G'_\bt)$ in general. This explains why $v^\flat_\bt$ does not guarantee the existence of an orientation between $G'$ and $\Psi'$ over $S'.$
}
\end{sremark}

\section{The case of a LG ring}

\subsection{LG rings}\label{subsec_LG} A ring $R$ is said to be a {\em $LG$-ring\/} if for every $n\in \NN$ and every $f\in R[X_1, \ldots, X_n]$ the polynomial $f$ represents a unit if and only if one of the following obviously equivalent conditions hold:

\begin{enumerate}
 \item\label{LG-defi} $f$ represents a unit over every localization $R_\m$, $\m$ a maximal ideal of $R$;

 \item\label{LG-defii} $f$ represents a unit over every field $R/\m$, $\m$ a maximal ideal of $R$;

  \item\label{LG-defiii} $f$ represents a unit over every $R$--field $F$;

 \item \label{LG-defiv} $\sum_{(r_1, \ldots, r_n) \in R^n} \, f(r_1, \ldots, r_n) R = R$,
 \end{enumerate}
 
Clearly a semilocal ring is a LG ring. An important fact is that locally free modules 
of rank $n$ over $LG$-rings are free  (\cite{EG} Thm. 2.10).
One says that a ring $R$ {\em satisfies the primitive criterion\/} \cite{EG, MW}
the following equivalent conditions hold: 

\begin{enumerate}
\item for every primitive polynomial $P\in R[X]$ there exists $r\in R$ such that $P(r) \in R^\times$;

\item for every primitive $Q\in R[X_1, \ldots, X_n]$ there exists $(r_1, \ldots, r_n)\in R^n$ such that $Q(r_1, \ldots, r_n) \in R^\times$;

\item $R$ is LG and all residue fields are infinite.
\end{enumerate}

In particular a semilocal ring with infinite residue fields satisfies the primitive criterion. 
The ring of all algebraic integers and the ring of all real algebraic integers  also satisfy the primitive criterion (\cite{EG} Cor. 4.7 and Cor. 6.4).
For more on LG rings, see \cite[\S 11.20]{GPR}.

\subsection{Schematically dominant morphisms} \label{subsec_dominant}

We recall that a morphism of schemes $f: Y \to X$
is schematically dominant if $\cO_X \to f_* \cO_Y$ is
injective; if $f$ is an immersion we say that 
$X$ is schematically dense in $Y$ \cite[11.10.2]{EGA4}.
If $g: X' \to X$ is flat,
and $f$ is schematically dominant and quasi-compact, then 
$f': Y \times_X X' \to X'$ is  schematically dominant
({\it loc. cit.}, 11.10.5).

If  $f: Y \to X$ is an $S$--morphism of schemes, we say that 
$f$ is schematically $S$-dominant (or universally 
schematically dominant with respect to $S$) if for each $S$-scheme $S'$, 
$f_{S'}: Y_{S'} \to X_{S'}$ is schematically 
dominant; similarly if $f$ is an immersion, we say that
$Y$ is schematically $S$--dense in $X$. 
Under mild assumptions, this property can be checked fiberwise ({\it ibid}, 11.10.9).
The base change property above extends mechanically:
If $g: X' \to X$ is flat  and
$f: Y \to X$ is schematically  $S$-dominant and quasi-compact, then 
$f': Y \times_X X' \to X'$ is schematically  $S$-dominant.

Rydh extended that to morphism of algebraic spaces \cite[\S 7.5]{Ry1}. A morphism of algebraic spaces $f: Y \to X$
is schematically dominant if the morphism
$\cO_X \to f_* \cO_Y$ is
injective in the small \'etale site;
if $f$ is an immersion, we say that $X$ is schematically 
dense in $Y$.

As in \cite[\S 3.5]{Ry2}  the base change property 
extends in that framework:
given a flat  morphism $g: X' \to X$
of algebraic spaces 
assuming that 
$f: Y \to X$ is schematically  dominant and quasi-compact, then  $f': Y \times_X X' \to X'$ is schematically dominant.

If $S$ is a scheme and $f: Y \to X$ is a morphism
of $S$-algebraic spaces, we say that 
$f: Y \to X$ is schematically $S$--dominant  
if for each $S$-scheme $S'$, 
$f_{S'}: Y_{S'} \to X_{S'}$ is schematically 
dominant; similarly if $f$ is an immersion, we say that
$Y$ is schematically $S$--dense in $X$.
One last time, given a flat  morphism $g: X' \to X$  and
assuming that 
$f: Y \to X$ is schematically  $S$-dominant and quasi-compact, then  $f': Y \times_X X' \to X'$ is schematically $S$--dominant.

\subsection{(4)-Versal torsors} \label{subsec_versal}

Our goal is to adapt the framework of versal torsors \cite[\S 5]{GMS} from fields to the LG
ring setting by using algebraic spaces for which we use 
Olsson's book as reference \cite{O} as well as the 
Stacks project \cite{St}. It is close to 
Reichstein-Tossici's definition of (3)-versality
for algebraic groups \cite[Def.\ 1.3]{RT}.

Let $S=\Spec(R)$ be the spectrum of a LG-ring $R$
and let $G$ be an affine flat $S$--group scheme of finite presentation.

\begin{sdefinition} 
Let $f:E \to X$ be a $G$--torsor where $E$ is a
quasi-compact quasi-separated $S$--scheme and $X$ is a quasi-compact quasi-separated  $S$-algebraic space.
We say that $E \to X$ is (4)--versal if it satisfies
the following property:
for each open retrocompact subspace $U$ 
schematically $S$--dense in  $X$, for each  $R$--algebra $B$ 
satisfying the primitive criterion, and for each $G$--torsor $F$ over $B$, there exists $x \in U(B)$ such that $E_x \cong F$ as $G$--torsors.

\end{sdefinition}

\begin{sremarks}
{\rm

\noindent (a)  We observe that if $f:E \to X$ is a (4)--versal  $G$--torsor so is
$f^{-1}(V) \to V$ for each schematically $S$--dense  open retrocompact subspace $V$ of $X$.

\smallskip

\noindent (b) There exists a dense open 
subspace $X'$ of $X$ that is a scheme \cite[Tag 06NH]{St}. If $X'$ can be chosen furthermore schematically $S$-dense
retrocompact, 
we see then that
the open dense retrocompact subscheme $U'=U \times_X X'$ can be chosen in the definition. 
This is the case if $S$ is noetherian
of dimension $\leq 1$ and $X$ is reduced separated
according to \cite[Tag 0ADD]{St}.

}
\end{sremarks}

\begin{slemma} \label{lem_versal} Let  $G \to \GL(\cE)$  be a faithful linear representation of
$G$ where $\cE$ is a locally free $R$-module of finite rank.

\smallskip

\noindent (1) The fppf quotient  $X= \GL(\cE)/G$ is representable
by an algebraic space  and 
 the quotient map $\GL(\cE) \to X$ is a (4)-versal $G$--torsor with right $G$-action.

\smallskip

\noindent (2) Assume that the vector $S$--group $\VV(\cE)$
(defined in the notation) admits an open universally dense $S$-subscheme $V$  which is $G$--stable  and such that $G$ acts freely on $V$. Let $Y=V/G$ be the quotient algebraic space.
Then $V \to V/G$ is a (4)--versal $G$--torsor.

\end{slemma}

\begin{proof}
\noindent(1)
If $R$ is noetherian, then the fppf quotient
$X=\GL(\cE)/G$ is represented by an 
$S$--algebraic space \cite[th.\ 3.1.1]{A}.
The usual noetherian reduction trick shows that it is the case in general.
Let $B$ be an $R$--ring satisfying the primitive criterion
and $F$ be a $G$-torsor over $B$.

We assume firstly that $\cE$ is of constant rank $n$ so that 
$\cE \cong R^n$ according to \S  \ref{subsec_LG}.
Consider the exact sequence of pointed sets:
\begin{align*}
1\to G(B)\to \GL_n(B)\to X(B)\overset{\delta}{\to} H^1_{fppf}(B,G)\to H^1_{fppf}(B,\GL_n).
\end{align*}
As $B$ is an  LG ring,  we have  $H^1_{fppf}(B,\GL_n)=1$ from the same references.
Hence the torsor $F$ is a pull-back of $\GL(\cE)\to X$ at some element $x\in X(B)$.

Note that we have a natural left $\GL_n$-action on $X$.
Hence $x$ defines a morphism $f_x$ from $\GL_{n,B}$ to $X_B$ as follows:
$$f_x(g)=g\cdot x,$$
for all $B$-algebra $C$ and all $g\in \GL_{n,B}(C)$.

Let $U$ be an open retrocompact subspace of $X$  that is
schematically $S$--dense in $X$. 
We have the following cartesian diagram
$$\begin{CD}
U\underset{X,f_x}{\times}\GL_{n,B} @>>>  \GL_{n,B}\\
@VVV    @VV f_x V\\
U_B @>>> X_B
\end{CD}
$$
The algebraic space
$U\underset{X,f_x}{\times} \GL_{n,B}$ is representable 
by an open retrocompact subscheme $V_x$ of $\GL_{n,B}$ which is $S$-dense according to base change property listed in \ref{subsec_dominant}. 
Since $\GL_{n, R}$ is $S$-dense in the affine space $\VV(\Mat_n(R))$,
 $V_x$ is $S$-dense in $\VV(\Mat_n(R))$,
According to \cite[prop.\ 1.4]{GN}, we have that $V_x(B) \not = \emptyset$.
In other words, there exists $g \in \GL_n(B)$ such that 
$g\cdot x$ is in the image of $U(B)$.
 Since $\delta(x)=\delta(g\cdot x)$, $F$ is a pull back of $\GL(\cE)\to X$ at  $g\cdot x\in U(B)$.

The reduction to the constant rank case is standard.
Since the rank function is locally constant we have
a decomposition $R=R_1 \times \dots \times R_c$ such that
$\cE_{R_i}$ has constant rank $n_i$ for $i=1,...,c$.
The point is that each $R_i$ is an LG ring and similarly
each $B_i=B \otimes_R R_i$ satisfies the primitive criterion 
\cite[Ex.\ 1.2.(b)]{GN}.
Since $G(B)= G(B_1) \times \dots \times G(B_c)$, 
the first case is enough to complete the proof.

\smallskip

\noindent (2)
Let $B$ be a $R$--ring satisfying the primitive criterion.
Let $U$ be an open retrocompact subspace of $V/G$ which is 
$S$--schematically dense.
We consider the $G$--torsor $P= V \times_{V/G} U$
over $U$ and want to show  that it is 
(4)-versal. It is enough to show that for a given  $G$-torsor  $F$ over $B$,
we can find $u \in U(B)$ such that $P_u  \cong F$.
The  Chinese Remainder theorem 
 implies that $V(B) \not = \emptyset$.
We pick  $v\in V(B)$, we can define a right $G$-equivariant morphism $f_v$ from 
$\GL(\cE)$ to $\VV(\cE)$ by 
$$f_v(\sigma)=v\circ\sigma,$$
where  $\sigma$ is an element of  $\GL(\cE)(C)$ and $C$ is a $B$-algebra. We have then a commutative diagram of $G$-torsors
$$\begin{CD}
\GL(\cE)_B@>f_v>>  V_B\\
@V{q}VV    @V{q'}V V\\
\GL(\cE)_B/G_B @>\overline{f}_v>> (V/G)_B.
\end{CD}
$$
The point is that 
$\overline{f}_v^{-1}(U_B)$ is  retrocompact and 
schematically $B$-dense in $X_B=\GL(\cE)_B/G_B$
according to the preliminaries of \ref{subsec_dominant}.

According to  (1) applied to $B$, the 
$G$-torsor $\GL(\cE)\to X$ is a (4)-versal torsor.
There exists $x \in \overline{f}_v^{-1}(U_B)$ such that 
$q^{-1}(x)\cong F$. We put $u= \overline{f}_v(x)$
and conlude that $P_u={q'}^{-1}(u) \cong F$.
Thus $V\to V/G$ is also a (4)-versal torsor.
\end{proof}

In this paper we shall use only the next special case.

\begin{slemma}\label{lem_versal_finite}
Let $G$ be  an $S$--group scheme that is  finite locally free. Let $\cE$ be a $G-R$-module which is locally free of finite rank as an $R$--module.
Let $U \subset \VV(\cE)$ be a $G$--stable 
retrocompact $S$-dense open subscheme such that $G$ acts freely on
$U$. Then  $U/G$ is representable by an $S$--scheme $X$ and the quotient map $U \to X$
is a versal (4)--torsor.
\end{slemma}

\begin{proof} According to \cite[III.2.6.1, Corollaire]{DG}, $U/G$ is representable
by an $S$-scheme $X$ and $U \to X$ is a $G$--torsor. Lemma \ref{lem_versal}
shows that $U \to X$ is a versal (4)--torsor.
\end{proof}

\begin{sremark}{\rm
In this setting, it is not clear whether (4)-versal torsors always exist.
}
\end{sremark}

\subsection{Embbedings in quasi-split groups}
In this section, we assume that $S$ is the spectrum of 
a ring $R$ satisfying the primitive criterion.
Let $G$ be a quasi-split reductive $S$--group scheme.

Let $(B,T)$ be a Killing couple of $G$. 
We identify the Weyl group $W_G(T)$ with $N_G(T)/T$ and denote by $\Dyn(G)$
the Dynkin scheme (which is finite \'etale over $S$). We have a canonical isomorphism
$\Dyn(\Phi(G,T)) \simlgr \Dyn(G)$.

According to \cite[XXIV.3.13]{SGA3}, if $G$ is simply connected, we have an isomorphism $T \simlgr R_{\Dyn(G)/S} (\GG_{m, \Dyn(G)})$.
We denote by $T^{reg}$ the open subscheme of regular elements of $T$.
Then  $W_G(T)$ stabilizes $T^{reg}$ and acts freely on it.
According to \cite[III.2.6.1, Corollaire]{DG}, the fppf sheaf $T^{reg}/W_G(T)$ is representable by a scheme
and the map  $\pi:T^{reg} \to T^{reg}/W_G(T)$ is  a $W_G(T)$--torsor.

Denote by $G//G$ the adjoint quotient. We use now  Chevalley's theorem \cite{L2} in this setting
$u: T/W_G(T) \simlgr G//G$.
Away from type $A_{2n}$, we can deal with
the Steinberg cross section $c: G//G \to G$ of the quotient map  $G \to G//G$ \cite[th.\ 5.7]{L2}; in that case
we can associate to an element
 $x \in (T^{reg}/W_G(T))(S)$  the semisimple regular  element
 $g =c( u(x)) \in G(S)$  giving rise to the maximal  $S$--subtorus $T_x=C_G(g)$ of $G$.
The following generalizes to the LG 
setting a result of Gille-Raghunathan \cite{Gi2004,Rg}.

\begin{stheorem} \label{thm_semi_local} Let $S$ be the spectrum of 
a ring $R$ satisfying the primitive criterion.
Let $G$ be a quasi-split reductive $S$--group scheme.
Let $\Psi$ be a root datum that is admissible for
$G$ and let $v \in \uIsomext( G, \Psi)(S)$ be an orientation. 
Then $\gE(G,\Psi,v)(S) \not = \emptyset$.
\end{stheorem}

\begin{sremark}\label{rem_finite} {\rm The statement holds also in the case when $S$ is the spectrum
of a  finite field. This is obvious since all homogeneous $G$-spaces have a $k$-rational point \cite[III.2.4]{Se}.
 }
\end{sremark}

\begin{proof}
Using the same kind of reduction as in the proof of Lemma 
\ref{lem_versal}.(1) we can assume without loss of generality that $G$ is of constant type.
It is harmless to assume furthermore that 
$G$ is semisimple adjoint by Remark \ref{rem_behaviour}. Furthermore using the 
isotypic  decomposition \eqref{dec3} for embedding functors, we can assume
that the geometric fibers of $G$ are almost  simple (note that a finite \'etale extension of $R$
satisfies the primitive criterion in view of \cite[Ex.\ 1.2.(b)]{GN}).
Let $(B,T)$ be a Killing couple of $G$. We denote  by $W_G(T)=N_G(T)/T$ the Weyl group and by $\Dyn(G)$
the Dynkin scheme (which is finite \'etale). Note that $W_G(T)\simeq W(\Phi(G,T))^{op}$. We set $W=W(\Phi(G,T))$.

\smallskip

Let $c\in \uIsomext(G,\Phi(G,T))(S)$ be the orientation coming from the natural inclusion $T\hookrightarrow G$, and denote by $u\in\uIsomext(\Phi(G,T),\Psi)(S)$ the image of $(c,v)$
 under the morphism in Prop.\ \ref{pairing}.
We consider the  right $W$--torsor $J=\uIsomint_u(\Phi(G,T),\Psi)$.
Note that $J$ can also be regarded as a left $W_G(T)$-torsor.
According to Lemma \ref{lem_compatibilities},
we have $\gE(G,\Psi,v)=J  \land^W\gE(G,\Phi(G,T),c)$.

By Lemma \ref{lem_versal_finite}, $\pi: T^{reg} \to  T^{reg}/W_G(T)$ is a (4)-versal (left) $W_G(T)$--torsor.
In particular there exists $x \in (T^{reg}/W_G(T))(S)$ such that
$\pi^{-1}(x) \cong J$ (as left $W_G(T)$--torsors and right $W$-torsors).

\medskip

\noindent
\textbf{Case 1.}  $G$ is  not of $A_{2n}$-type.
By Steinberg cross section, there is a semisimple regular element $t\in G(S)$ mapped to $x$.
The centralizer of $t$ is a maximal torus $T'$ of $G$ over $S$.
Let $c'\in\uIsomext(G,\Phi(G,T'))(S)$ be the canonical orientation induced by the natural inclusion from $T'$ to $G$.
Denote by $u'\in\uIsomext(\Phi(G,T),\Phi(G,T'))(S)$ be the image of $(c,c')$ under 
the pairing in Prop.\ \ref{pairing}.
By Prop.\ \ref{W-torsor} the  left $W_G(T)$-torsors $\pi^{-1}(x)$ and $\uIsomint_{u'}(\Phi(G,T),\Phi(G,T'))$ are isomorphic.
It follows that  $J \land^W \gE(G,\Phi(G,T),c) =\gE(G,\Psi,v)$
is isomorphic to $\pi^{-1}(x)\land^W\gE(G,\Phi(G,T),c)  \cong  \gE(G,\Phi(G,T'),c')$.
As  $\gE(G,\Phi(G,T'),c')(S)  \not = \emptyset$, we conclude that
$\gE(G,\Psi,v)(S)$ is not empty.

\medskip

\noindent
\textbf{Case 2.} $G$ is of $A_{2n}$-type. 
The key point is to reduce to the case  of type $A_{2n+1}$. 
Let $R_\sharp/R$ be the \'etale quadratic algebra corresponding to the outer type of 
$G$. Then $R_\sharp/R$ splits $G$ and we denote by $\sigma$ its canonical involution. 

We consider the $R_\sharp$-module $V= (R_\sharp)^{2n+1}=V_{-} \oplus R_\sharp \oplus V_{+}$
with $V_{-}=(R_\sharp)^{n}$ and $V_{+}=(R_\sharp)^{n}$.
It is equipped with  the hermitian form $h(x,y)= \sigma(x_0) \, y_0+ \sum\limits_{i=1,..,n} \sigma(x_{-i})\,  y_i$ with the notation $x=(x_i)_{i=-n,..,n}$.
Denoting by $\HH_n$ the standard hyperbolic space of rank $2n$, we
have an orthogonal decomposition $h= \HH_n \perp \langle 1 \rangle$.

The unitary group scheme $U(V,h)$ is reductive 
and its derived group $SU(V,h)$
is a semisimple simply connected $R$--group scheme
of type  $A_{2n+1}$ \cite[prop.\ C.6]{GN}. We consider the cocharacter $\lambda: \GG_{m,R} \to U(V,h)$
defined by $\lambda(t)= \mathrm{diag}(t^{-n},t^{-n+1}, \dots, t^{-1}, 1, t, \dots, t^{n-1}, t^{n})$.
The associated parabolic $R$--subgroup $P=P_{U(V,h)}(\lambda)$ 
carries the Levi subgroup $C_{U(V,h)}(\lambda)$ which is the  $R$--torus
$E=U(V,h) \cap R_{R_\sharp/R}( \GG_m^{2n+1})\cong \GG_m \times_R R_{R_\sharp/R}(\GG_m^n)$ so that $P$ is an 
$R$--Borel subgroup of $U(V,h)$. 
In view of \cite[\S 2.12]{GN}, the $R$--group scheme $U(V,h)$ is then quasi-split.
By considering $E$, we see that the outer type of $SU(V,h)$ is 
$[R_\sharp] \in H^1(R, \ZZ/2\ZZ)$
so that $G \cong SU(V,h)$. 
The group $G$ carries then a Cartan involution defined over $R$, hence if we  have an embedding with respect to one orientation, we have an embedding with respect to the other orientation.
Hence it suffices to show that $\gE(G,\Psi)(R)\neq\emptyset$.

Since $U(V,h)$ and $G=SU(V,h)$ have same automorphism group, 
same scheme of maximal tori and same isomorphism classes of root data,
we can work with $\widetilde G=U(V,h)$.
Let $\gS_{i}$ be the symmetric group of the set $\{1,...,i \}$. We have a natural monomorphism from $\gS_{2n+1}$ to $\gS_{2n+2}$ and hence a monomorphism $\iota$ from
$\gS_{2n+1}\times\ZZ /2 \ZZ$ to $\gS_{2n+1}\times \ZZ/ 2\ZZ$ 
that maps $\ZZ/2\ZZ$ isomorphically to itself.
The map $\iota$ induces a map $\iota_{\ast}: H^1_{\acute{e}t}(R,\gS_{2n+1}\times \ZZ/2\ZZ )\to H^1_{\acute{e}t}(R,\gS_{2n+2}\times\ZZ/ 2\ZZ)$.
Note that $\gS_{2n+1}\times\ZZ /2 \ZZ$ is isomorphic to the Weyl group of the split root datum of $\GL_{2n}$. 
Hence $H^1_{\acute{e}t}(R,\gS_{2n+1}\times \ZZ/2\ZZ)$ classifies those root data that are \'etale locally isomorphic to the split root datum of $\GL_{2n+1,R}$.

Thus $\Psi=(\cM,\cM^\vee,\cR,\cR^\vee)$ corresponds to some   $[\alpha] \in H^1_{\acute{e}t}(R,\gS_{2n+1}\times \ZZ/2\ZZ )$.
We twist the split root datum of $\GL_{2n+2,R}$ by $\iota_\ast(\alpha)$ and denote it by $\Psi'=(\cM',\cM'^\vee,\cR',\cR'^\vee)$.
Let $T$ and $T'$ be the tori determined by $\cM$ and $\cM'$ respectively.
Note that by our construction $\cM'^\vee$ can be written as $\cM^\vee\oplus \cE^\vee$ with $\cE$ is a $fppf$ twisted form of $\ZZ$ in such a way that  $R^\vee$ is injectively sent to a subset of ${R'}^\vee$.
Fix such an isomorphism $\cM'^\vee\simeq\cM^\vee\oplus \cE^\vee$. This in turns gives an isomorphism $T'\simeq T\times \bR^{(1)}_{R_\sharp/R}(\bG_m)$.
We denote by $T_0$ the subtorus $1\times \bR^{(1)}_{R_\sharp/R}(\bG_m)\subseteq T'$ (under the above isomorphism).

We put $V'= V \oplus R^\sharp$ and $h'\bigl( (v,r), (w,s) \bigr)
= h(v,w) - \sigma(r) s$. We have an orthogonal decomposition $(V',h')= \HH(V_{-}) \oplus (R_\sharp^2, h'')$
where $h''((r_1,s_1), (r_2,s_2))=  {\sigma(r_1) s_1 -\sigma(r_2) s_2}$.
Since $(1,1)$ is an isotropic unimodular vector
of  $h''$, we know that $h''$ is hyperbolic in view of \cite[prop.\ 3.7.1]{K}.
Since $R_\sharp$ is a LG ring, the locally free $R_\sharp$--modules of constant rank are free so that $h'' \cong \HH_1$ and $h' \cong \HH_{n+1}$.
The $R$--group scheme $U(V',h')$ is then a
quasi-split semisimple $R$--group scheme of type $A_{2n+1}$.

By our assumption, there is an orientation between $\Psi$ and $G$. 
Hence $\Psi$ and $G$ are both become inner forms after base change to $R$.
By our construction of $\Psi'$ and $G'$, they are also both of inner form after base change to $R$.
Hence there is an orientation $v'$ between $\Psi'$ and $G'$. 
Since $G'$ is quasi-split of type $A_{2n+1}$, by Case 1 we have $\gE(G',\Psi',v')(R)\neq\emptyset$.

Choose $f\in \gE(G',\Psi',v')(R)$. 
Consider the subgroup $f(T_0)$ in $G'$.
Let $C$ be the centralizer $\uCentr_{G'}(f (T_0))$, which is a Levi subgroup of $G'$. 
Clearly $f(T')\subseteq C$.
Let $V_0$ be the maximal $R$-submodule of $V$ fixed by $f(T_0)$. Namely, $V_0$ is the maximal $R$-submodule on which $f(T_0)$ acts trivially.
Since the action of $f(T_0)$ is $R_\sharp$-linear, 
$V_0$ is furthermore an  $R_\sharp$-submodule of $V$.
It can be checked  \'etale locally that $V=V_0 \perp V_0^\perp$
and that the $R_\sharp$--module $V_0$ is locally free of rank $2n+1$.
Let $h_0$ be the restriction of $h'$ on $V_0$.

\begin{claim} $h_0 \cong \HH_n \perp \langle a \rangle$ for some $a \in R^\times$.
\end{claim}

Since $R_\sharp$ is a LG ring, the $R_\sharp$--module $V_0$ is free of rank $2n+1$
and $V_0^\sharp$ is free of rank $1$.
We have then an orthogonal decomposition $\HH_{n+1}=h'= h_0 \perp \langle - a \rangle$
for some $a \in R^\times$. The same argument as above
shows that  $\langle a \rangle \perp \langle - a \rangle \cong\HH_1$
so that $h_0 \perp \HH \cong \HH_{n+1} \perp \langle a \rangle$.
The cancellation property \cite[App.\ 5.5]{GN} yields that
$h_0 \cong \HH_n \perp \langle a \rangle$ as desired.
The Claim is established.

The same reasoning as above shows that $U(V_0,h_0)$ 
is the quasi-split $R$--group scheme of type $A_{2n}$ of
Tits invariant $[R_\sharp] \in H^1(R,\ZZ/2\ZZ)$.
In other words we have $U(V_0,h_0)\simeq U(V,h)=\widetilde G$.  

Clearly $C$ stabilizes $V'$ and hence induces an action of $C/f(T_0)$ on  $(V',h')$.
This gives a group homomorphism of $C/f(T_0)$ to $U(V_0,h_0)$, which is an isomorphism at each geometric point $\ol{s}$ of $S$.
Therefore $C/f(T_0)\simeq U(V_0,h_0)$ and $f$ gives an embedding of $T\simeq T'/T_0$ in $\widetilde G$ with respect to $\Psi$.
\end{proof}

\begin{scorollary} \label{cor_main} Let $S$ be a scheme 
that is the spectrum of a ring $R$ satisfying the primitive criterion.
Let $G$ be a quasi-split reductive $S$--group scheme.
Let $X$ be an $S$-scheme
satisfying one of the following conditions:

\smallskip 

(i) $S$ is noetherian;

\smallskip

 (ii) $X$ is locally of finite type over $S$.

\smallskip

We assume that $X$ is a left homogeneous space under $G$
and that for each point $s \in S$, the stabilizers of $X_{\ol s}$ are maximal tori  of $G_{\ol s}$.
Then $X \cong G/T$ where $T$ is a maximal $S$--torus of $G$.
\end{scorollary}

\begin{proof}
It follows immediately from Theorem \ref{thm_main} and Theorem \ref{thm_semi_local}. 
\end{proof}

\begin{sremark}\label{rem_finite2} {\rm The statement holds also in the case $S$ is the spectrum
of a  finite field in view of Remark \ref{rem_finite}.
 }
\end{sremark}

\subsection{An application to the local-global principle for embeddings}

In this section we assume that $k$ is a global field with characteristic not $2$.
Let $K$ be a field, $A$ be a central simple algebra over $K$ with
involution $\tau$, and $E$ be an \'etale algebra over $K$ with
involution $\sigma$. Suppose that $\tau|_K=\sigma|_K$ and that $k$ is
the field of invariants $K^{\tau}$.

Let $\dim_L A=n^2$.
Assume that $\dim_K E= n$, and that if $K\neq k$, then $\dim_k E^\sigma=n$.
If $K=k$, then we assume that $\dim_k E^\sigma =[\frac{n+1}{2}]$. 

The local-global principle for embeddings of $K$-algebras with involution $(E,\sigma)$ into $(A,\tau)$ is first considered in \cite{PR1}.
One can associate a root datum $\Psi$ to $(E,\sigma)$ and a reductive group $U(A,\tau)^\circ$ to $(A,\tau)$, where $U(A,\tau)^\circ$ is the connected component of the unitary group $U(A,\tau)$. (See \cite{L1} \S 1.3.2.)
In \cite{L1} \S 1.3.2 and \cite{BLP1} \S 1.2 and 1.3, it is shown that there is a bijection between the embeddings of algebras with involutions over $k$ and  the embeddings of corresponding root data into reductive groups.

Combining with Theorem \ref{thm_semi_local}, we have the following theorem.
\begin{theorem}\label{embed_algebras}
Let $(E,\sigma)$ and $(A,\tau)$ be as above.  If $U(A,\tau)^\circ$ is quasi-split, then the local-global principle holds for the $K$-algebra embeddings of $(E,\sigma)$ into $(A,\tau)$.
\end{theorem}
\begin{proof}
Let $\Psi$ be the root datum associated to $(E,\sigma)$ (\cite{L1} \S 1.3.2).
Let $G=U(A,\tau)^\circ$. By Theorem \ref{thm_semi_local}, it suffices to show that $\uIsomext(G,\Psi)(k)\neq \emptyset$.

First we consider the case where $K\neq k$.
In this case, we can always fix an orientation by \cite{L1} Theorem 1.15 (2) and Remark 1.16.

Next we consider the case where $K=k$.
If $\tau$ is symplectic involution or $\tau$ is orthogonal with $n$ odd, then $G$ is of type $C_{n/2}$ or $B_{(n+1)/2}$.
Then there is always an orientation between $\Psi$ and $G$.

Consider the case that $\tau$ is orthogonal with $n$ even.
Suppose there is  an embedding $\eta_v$ of $(E\times_k k_v,\sigma\otimes \id_{k_v})$ into $(A\times_k k_v,\sigma\otimes\id_{k_v})$ for all places $v$ of $k$.
Then $\eta_v$ gives an isomorphism between the discriminant $\Delta(E)$  of $E$ and the center $Z(A,\tau)$ of the Clifford algebra of $(A,\tau)$ over $k_v$ (see \cite{BLP2} \S 2.3).
As they are both quadratic \'etale algebras over $k$, this means that $\Delta(E)$ is isomorphic to $Z(A,\tau)$ over $k$.
By \cite{BLP1} Proposition 1.3.1, an isomorphism between $\Delta(E)$ and $Z(A,\tau)$ gives an orientation between $\Psi$ and $G$.

Hence for $K=k$, we can always fix an orientation between $\Psi$ and $G$.
\end{proof}

\bigskip

\medskip


\bigskip

\begin{thebibliography}{99}

\bibitem{A}  S.~Anantharaman, {\it Sch\'emas en groupes, espaces homogènes et espaces alg\'ebriques sur une base de dimension 1},
  M\'emoires de la Soci\'et\'e Math\'ematique de France {\bf 33} (1973), 5-79. 


\bibitem{BLP1} E.~Bayer-Fluckiger, T.-Y.~Lee and R.~Parimala, {\it Embedding functor for classical groups and Brauer-Manin obstruction},
Pac. J. Math. \textbf{279} (2015), 87-100.

 \bibitem {BLP2} E.~Bayer-Fluckiger, T.-Y.~Lee, and R.~Parimala, \textit{Embeddings of maxi-mal tori in classical groups and explicit Brauer-Manin obstruction}, 
 J. Eur. Math. Soc. {\bf 20} (2018), 137-163.
 
 
 
\bibitem{Bo}  A.~Borel, {\it Linear algebraic groups}, Graduate Texts in Mathematics {\bf  126} (2nd ed.),
Berlin, New York: Springer-Verlag.

\bibitem{BLR} S.~Bosch, W.~L\"utkebohmert, M.~Raynaud,
{\it N\'eron models}, Ergebnisse der Mathematik und ihrer Grenzgebiete
{\bf 21} (1990), Springer.




\bibitem{BT} F.~Bruhat, J.~Tits, {\it
Groupes r\'eductifs sur un corps local : II. Sch\'emas en groupes.
Existence d'une donn\'ee radicielle valu\'ee},
Publications Math\'ematiques de l'IH\'ES {\bf 60} (1984), 5-184.


\bibitem{CF} B.~Calm\`es, J.~Fasel, {\it Groupes classiques},
     Autour des sch\'emas en groupes, vol II, Panoramas et Synth\`eses {\bf 46} (2015),    1-133.












\bibitem{DG} M.~Demazure, P.~Gabriel,
{\it Groupes alg\'ebriques}, Masson (1970).


\bibitem{EG} D.~R.~Estes and R.~M.~Guralnick, {\em Module equivalences: Local to global when primitive polynomials represent units}, J.~Algebra \textbf{77} (1982), 138-157.


\bibitem{GMS} S.~Garibaldi, A.~Merkurjev and J.-P.~Serre, {\it
    Cohomological invariants
 in Galois cohomology}, University Lecture Series \textbf{28} (2003),
American Mathematical Society.


\bibitem{GPR}  S.~Garibaldi,  H.P.~Petersson, M.~Racine, 
{\it  Albert algebras over commutative rings}, to appear
in  Cambridge University Press as New Mathematical Monographs vol.\ {\bf 48} (2024).


\bibitem{EGA1} A.~Grothendieck, J.-A.~Dieudonn\'e, {\it El\'ements
de g\'eom\'etrie alg\'ebrique. I}, Grundlehren der Mathematischen Wissenschaften  166; Springer-Verlag, Berlin, 1971.




\bibitem{EGA4} A.~Grothendieck (avec la collaboration de J.~Dieudonn\'e),
{\it El\'ements de G\'eom\'etrie Alg\'ebrique IV}, Publications math\'ematiques de l'I.H.\'E.S. no 20, 24, 28 and 32 (1964 - 1967).






\bibitem{Gi2004} P.~Gille, {\it  Type des tores maximaux des groupes semi-simples}, J. Ramanujan Math. Soc. {\bf 19} (2004), 213-230.


\bibitem{Gi2015} P.~Gille, {\it Sur la classification des
sch\'emas en groupes semi-simples}, ``Autour des sch\'emas en groupes, III'',
Panoramas et Synth\`eses {\bf 47} (2015), 39-110.


\bibitem{GN} P.~Gille, E.~Neher, {\it Group schemes over LG-rings  and applications to cancellation theorems and Azumaya algebras}, 
https://hal.science/hal-04631256


\bibitem{GNR} P.~Gille, E.~Neher, C.~ Ruether, {\it The Norm Functor over Schemes}, preprint arXiv:2401.15051







\bibitem{Gir} J.~Giraud,
{\em Cohomologie non-ab\'elienne}, Springer (1970).


  \bibitem{GW} U. G\"ortz, T. Wedhorn, {\it  Algebraic Geometry I: Schemes},
   Springer (2020). 

\bibitem{G}  A.~Grothendieck, {\it Le groupe de Brauer. I. Alg\`ebres d'Azumaya et interpr\'etations diverses}, Dix expos\'es sur la cohomologie des sch\'emas, 46-66,
Adv. Stud. Pure Math., 3, North-Holland, Amsterdam, 1968.


  
  
\bibitem{K} M.-A.~Knus, {\it Quadratic and Hermitian Forms over
    Rings}, Grundlehren der mathematischen Wissenschaften {\bf 294}
    (1991), Springer.






\bibitem{L1} T.-Y.~Lee, {\it Embedding functors and their arithmetic properties},
 Comment. Math. Helv. {\bf 89} (2014),  671-717.

\bibitem{L2} T.-Y.~Lee, {\it Adjoint quotients of reductive groups},
 Panoramas et  Synth\`eses {\bf 47} (2015), 131-145.



\bibitem{M} B.~Margaux, {\it 
 Vanishing of Hochschild cohomology for affine group schemes and rigidity of homomorphisms between algebraic groups}, 
Doc. Math. {\bf 14} (2009), 653-672.


\bibitem{MW} B.~R.~McDonald and W.~C.~Waterhouse, {\em Projective modules over rings with many units}, Proc.~Amer.~Math.~Soc.~\textbf{83}(3) (1981), 455-458.

\bibitem{O} M.~Olsson, {\it Algebraic spaces and stacks}, American Mathematical Society Colloquium Publications, 62. American Mathematical Society, Providence, RI, 2016.

\bibitem{PR1} G.~Prasad and A.S.~Rapinchuk, \textit{Local-Global principles for embedding of fields with involution into semple algebras with involution},
 Comment. Math. Helv. \textbf{85} (2010), 583-645.

\bibitem{Rg} M.S.~Raghunathan, {\it
Tori in quasi-split-groups},
J. Ramanujan Math. Soc. {\bf 19} (2004), 281-287.

\bibitem{R} M.~Raynaud, {\it  Faisceaux amples sur les sch\'emas en groupes et les espaces homog\`enes},
 Lecture Notes in Mathematics {\bf  119} (1970), Springer-Verlag, Berlin-New York.


 \bibitem{RT} Z.~Reichstein, D.~Tossici, {\it
Special groups, versality and the Grothendieck-Serre conjecture}, preprint (2019),
 arXiv:1912.08109

\bibitem{Ry1}  D.~Rydh, {\it Submersions and
effective descent of \'etale morphism},
Bull. Soc. math. France
{\bf 138}(2010)  181-230.

\bibitem{Ry2}  D.~Rydh,  {\it Approximation of Sheaves on Algebraic Stacks}, 
International Mathematics Research Notices (2016),  717-737.


\bibitem{SGA3} {\it S\'eminaire de G\'eom\'etrie
alg\'ebrique de l'I.H.\'E.S., 1963-1964, Sch\'emas en groupes, dirig\'e 
par M.\ Demazure et A.\ Grothendieck},  Lecture Notes in Math. 151-153.
Springer (1970).

\bibitem{SGA4} M.~Artin, A.~Grothendieck, J.-L.~Verdier, 
{\it Th\'eorie des topos et
Cohomologie \'etale des sch\'emas},  in Lecture Notes in Math., Nos. 269, 270, 305. Springer-
Verlag, Berlin/Heidelberg/New York, 1972-1973.



\bibitem{Se} J.-P.~Serre,   {\it Cohomologie
galoisienne}, 5-i\`eme version r\'evis\'ee, Lecture Notes
in Mathematics {\bf 5}, Springer.


\bibitem{St} Stacks project, https://stacks.math.columbia.edu



\end{thebibliography}
\end{document}